\documentclass[american, 11pt]{amsart}
\usepackage[utf8]{inputenc}
\usepackage{amsmath, amssymb, amsthm}
\usepackage{accents}
\usepackage{enumerate}
\usepackage[curve]{xypic}
\usepackage[all]{xy}

\usepackage{graphicx}
\usepackage{amscd}
\usepackage{graphicx}
\usepackage{hyperref}
\usepackage{comment}
\usepackage{xcolor}
\usepackage{soul}
\numberwithin{equation}{section} \theoremstyle{plain}
\usepackage{color}

\newcounter{warn}[page]
\newcommand{\danger}{${\color{red}\triangle}\llap{\raisebox{.3ex}%
{\tiny!\hspace{1.45ex}}}$}
\newcommand{\warning}[1]{%
\raisebox{.01em}[0em]{\danger\ifnum\value{warn} > 1%
\tiny\bf\arabic{warn}\fi}%
\marginpar{\color{red}\tiny{\ifnum\value{warn} > 1\tiny\bf\arabic{warn}:\fi}\tiny #1}%
\stepcounter{warn}}

\newtheorem{theorem}{Theorem}[section]
\newtheorem{definition}[theorem]{Definition}

\newtheorem{corollary}[theorem]{Corollary}
\newtheorem{proposition}[theorem]{Proposition}
\newtheorem{lemma}[theorem]{Lemma}

\theoremstyle{remark}
\newtheorem{remark}{Remark}

\newcommand{\Z}{\mathbb{Z}}

\newcommand{\R}{\mathbb{R}}

\renewcommand{\P}{\mathbb{P}}
\renewcommand{\S}{\mathbb{S}}
\newcommand{\T}{\mathbb{T}}
\renewcommand{\underbar}[1]{\underaccent{\bar}{#1}}

\newcommand{\eval}{\mathrm{ev}}
\newcommand{\Rel}[1][]{{\mathcal{R}_{#1}}}
\newcommand{\Crit}{\mathrm{Crit}}

\newcommand{\sublevel}[3][]{{#2}^{#1\leq#3}}       
\newcommand{\zipped}[2]{\lfloor#1\rfloor_{#2}}
\newcommand{\zip}[2]{\zeta_{#1}^{#2}}
\newcommand{\Deck}{\mathcal{D}}

\newcommand{\Cov}[1]{\tilde{#1}}
\newcommand{\UCov}[1]{\widetilde{#1}}
\newcommand{\Steps}[1][\alpha]{\mathfrak{G}_{#1}}
\newcommand{\Rels}[1][\alpha]{\mathfrak{R}_{#1}}
\newcommand{\PreSpan}[1]{<\Deck,#1>}
\newcommand{\Span}[1]{\overline{\PreSpan{#1}}}
\newcommand{\NormalSpan}[1]{\Span{#1}^{\ast}}
\newcommand{\PreFree}[2][\Deck]{F_{#1\times#2}}
\newcommand{\Free}[2][\Deck]{\overline{\PreFree[#1]{#2}}}
\newcommand{\freeprod}{\mathop{\ast}}
\newcommand{\mingen}{\mu_{DTC}}
\newcommand{\minrel}{\rho_{DTC}}
\newcommand{\abelian}[1]{#1^{\mathrm{ab}}}
\newcommand{\proabelian}[1]{{#1}^{\underline{\mathrm{ab}}}}

 \DeclareMathOperator\Hom{Hom}
\DeclareMathOperator\ProHom{{\underline{Hom}}}

\newcommand{\F}{{\mathbb F}}

\renewcommand{\H}{{\mathbb H}}

\renewcommand{\k}{{\mathfrak k}}

\newcommand{\h}{{\mathfrak h}}

\begin{document}

\title[Novikov fundamental group]{\large Novikov Fundamental group}%
\subjclass[2010]{Primary 57R19; Secondary 57R70, 57R17}

\author[J.-F. Barraud]{Jean-Fran\c cois Barraud $^+$}%
\email{barraud@math.univ-toulouse.fr}%
\address{%
  Institut de mathématiques de Toulouse\\
 Université Paul Sabatier -- Toulouse III\\
  118 route de Narbonne\\
  F-31062 Toulouse Cedex 9\\
  France} %
\thanks{$^+$ This work was partially supported by the CIMI}
\author[A. Gadbled]{Agnès Gadbled $^{++}$}%
\email{agnes.gadbled@math.uu.se}
\address{%
 Department of Mathematics,
Uppsala University, Box 480,
751 06 Uppsala, Sweden}%
 \thanks{$^{++}$ This work was partially supported by the
  Hausdorff Research Institute for Mathematics (HIM), University of Bonn
  and Wallenberg grant KAW 2016-0440}
\author[R. Golovko]{Roman Golovko $^{+++}$}%
\email{rgolovko@ulb.ac.be}%
\address{%
Départment de Mathématique, Université libre de Bruxelles, CP
218, Boulevard du Triomphe, B-1050 Bruxelles,
Belgique}%
 \thanks{$^{+++}$ This work was supported by the ERC Consolidator Grant 646649 ``SymplecticEinstein''}%
\author[H.V. L\^e]{H\^ong V\^an L\^e $^{++++}$}%
\email{hvle@math.cas.cz}%
\address{Institute of Mathematics CAS, Zitna 25,
  11567 Praha 1, Czech Republic} \thanks{$^{++++}$ Research of HVL  was supported by RVO:67985840 and the GA\v CR-project 18-00496S}


\date{\today}

\maketitle

$$
\includegraphics[width=2cm]{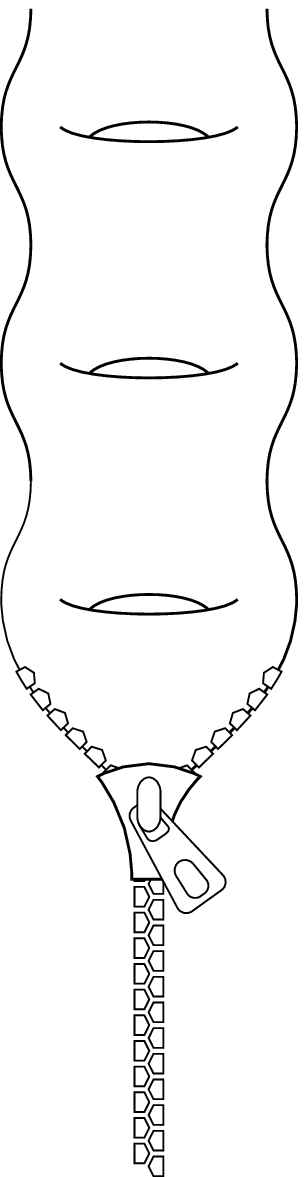}
$$

\begin{abstract}
  Given a $1$-cohomology class $u$ on a closed manifold $M$, we define a
  Novikov fundamental group associated to $u$, generalizing the usual
  fundamental group in the same spirit as Novikov homology generalizes
  Morse homology to the case of non exact $1$-forms.

  As an application, lower bounds for the minimal number of index $1$ and
  $2$ critical points of Morse closed $1$-forms are obtained,
  that are different in nature from those derived from the Novikov homology.
\end{abstract}

\section{Introduction and main results}

Consider a smooth and closed manifold $M$, and a  closed $1$-form
$\alpha$ in a non trivial cohomology class $u\in H^{1}(M,\R)$.

By analogy with the exact case, a point $p$ in $M$ is called
critical for $\alpha$ if $\alpha_{p}=0$. Such a critical point $p$
is said to be non degenerate, or Morse, if it is non degenerate as
the critical point of a local primitive $f_{\alpha}$ of $\alpha$ in
a neighborhood of $p$. The 1-form $\alpha$ itself is said to be
Morse if all its critical points are non degenerate.

The Novikov theory relates the critical points of (Morse) $1$-forms
$\alpha$ in the cohomology class $u$ and the ``topology'' of the
pair $(M,u)$.

\smallskip

The case $u=0$ is addressed by the Morse theory, which exhibits a
tight relation between the topology of $M$ and the critical points
of a Morse function on it~; in particular, Morse functions have to
have enough critical points to generate both $H_{*}(M)$ (the
classical Morse inequality) and $\pi_{1}(M)$ (\cite[Theorem 7.10, p.
172]{Sharko1993}).

Another extreme case is when $M$ fibers over the circle
$M\xrightarrow{\pi}\S^{1}$ and $\alpha=\pi^{*}(d\theta)$~: $\alpha$
obviously has no critical point in this case, regardless of the
actual topology of $M$. Notice that the converse is also true~: if
there is a $1$-form without critical points in $u$, then it comes
from a fibration over the circle \cite{Tischler1970}.

\medskip

The celebrated Novikov homology  \cite{Novikov} offers some
algebraic measurement of the complexity of the topology of $M$
``with respect to $u$'' or of the pair $(M,u)$, and the Novikov
inequalities \cite{Novikov} as well as further numerical invariants
developed in the same spirit \cite{Farber2004}, \cite{Pajitnov2006}
give lower bounds for the number of critical points of Morse
$1$-forms in the cohomology class $u$. As consequence, Novikov
proved the existence of periodic   solution in Kirchhoff  problem
\cite[Theorem 1, Appendix 1, p. 386]{DNF1984}. Novikov homology
also  enters in theory of symplectic  fixed points  on compact
symplectic manifolds \cite{LO1995} or Lagrangian
intersections/embeddings problems (\cite{Sikorav1986, Damian2009,
  Gadbled2009}).

\medskip

The object of this paper is to give a Novikov theoretic version of
the fundamental group, that plays in Novikov theory the role the
usual fundamental group plays in Morse theory.

More precisely, we fix a non trivial cohomology class $u\in
H^{1}(M,\R)$, and consider an integration cover $\Cov{M}$ associated
to $u$.

\begin{remark}\label{rk:ChoiceOfCoveringSpace}
  Throughout this paper, all the integration covers will always be
  supposed to be connected.

  The usual choices of covering spaces are either the minimal integration
  cover, or the universal cover, but any intermediate cover could be
  used~: each choice defines a different version of the invariant.

  The situation is similar to Novikov homology, which  was originally
  defined using the minimal integration cover \cite{Novikov} but was
  later extended to the universal cover (and in fact any intermediate
  cover) by Sikorav \cite{Sikorav1987}.
\end{remark}

We denote by $\tilde  u: \pi_{1}(M) \to \R$ the composition  of the
Hurewicz homomorphism $\pi_1 (M) \to H_{1}(M, \Z)$ with the
evaluation map $u: H_{1}(M, \Z) \to  \R$. Since
$\pi_{1}(\Cov{M})\subset\ker \tilde u$, the homomorphism $\tilde u$
descends to a homomorphism also denoted by $u: \Deck
=\pi_{1}(M)/\pi_{1}(\Cov{M}) \to \R$.  Note that $\Deck $ is the
deck transformation  group of $\Cov{M}$.

Then to $u$ and each choice of integration cover  $\Cov{M}$ is
associated a group $\pi_{1}(\Cov{M},u)$ endowed with an action of
$\Deck$, such that the following holds~:
\begin{theorem}\label{thm:Main}
  \begin{enumerate}
  \item
    For every choice of a closed $1$-form $\alpha$ in the class $u$ and a
    primitive $f_{\alpha}$ of $\alpha$, there is a group
    $\pi_{1}(f_{\alpha})$ that is isomorphic to $\pi_{1}(\Cov{M},u)$,

  \item
    there is a suitable notion of generators and relations ``up to deck
    transformations and a completion'' for which $\pi_{1}(\Cov{M},u)$ is
    finitely presented,
  \item
    if $\alpha$ is Morse, then the minimal number of generators (up to
    deck transformations and a completion) of $\pi_{1}(\Cov{M},u)$ is a
    lower bound for the number of index $1$ critical points of $\alpha$,

  \item
    similarly, if $\alpha$ is Morse, then the minimal number of relations
    (up to deck transformations and a completion) in $\pi_{1}(\Cov{M},u)$
    is a lower bound for the number of index $2$ critical points of
    $\alpha$.
  \end{enumerate}
\end{theorem}

\begin{remark}
  This definition does not use any extra choice of a base point. This is
  because in the non exact case (i.e. $u\neq 0$), which is the case we
  are interested in, there is a canonical choice of a base point, which
  essentially consists in putting it ``at $-\infty$''.

  In the exact case however, the primitives
  of $u$ are all bounded and this choice does not make sense anymore.

  Using other conventions about the base point to get a more uniform
  definition is possible but restrictive~: the usual notion of base point
  could be used, but breaks the action of the deck transformations which
  is an essential feature of the resulting group, or it could be replaced by
  a more subtle notion of base point (like properly embedded and shift
  invariant lines $\R\to\Cov{M}$) but this requires $u$ to be integral.

  Since we are only interested in the non exact case, we will stick to it
  and use the canonical choice of base point coming with it throughout
  this paper.

\end{remark}

To express the functoriality properties of the construction,
consider the Novikov fundamental group as associated to objects of
the form $(\Cov{M}\xrightarrow{\pi}M,u),$ where $\Cov{M}$ is a
connected covering space of a closed smooth manifold $M$ and $u$ is
a non trivial cohomology class in $H^{1}(M,\R)$ satisfying
$\pi^{*}(u)=0$.

A morphism between two such objects
$(\Cov{M}_{i}\xrightarrow{\pi_{i}}M_{i},u_{i})$, $i=1,2$, is then a
smooth covering map
$$
\xymatrix{
  \Cov{M}_{1}\ar^{\tilde{\varphi}}[r]\ar[d] & \Cov{M}_{2}\ar[d]\\
  M_{1}\ar^{\varphi}[r] & M_{2}\\
  }
$$
such that
$$
u_{1}=\varphi^{*}(u_{2}).
$$
\begin{theorem}\label{thm:funct}
  The Novikov fundamental group is functorial, i.e.
  \begin{enumerate}
  \item
    to an arrow $(\Cov{M}_{1} \to M_{1},u_{1}) \xrightarrow{\varphi}
    (\Cov{M}_{2} \to M_{2},u_{2})$ as above is associated a group
    homomorphism $\varphi_{*}: \pi_{1} ( \Cov{M_{1}}, u_{1}) \to \pi_{1}
    (\Cov{M}_{2}, u_{2})$  which commutes with the corresponding deck
    transformations,
  \item
    if $(\Cov{M}_{2}\to M_{2},u_{2}) \xrightarrow{\phi} (\Cov{M}_{3} \to
    M_{3},u_{3})$ is another morphism, then $[\phi \circ \varphi]_{*} =
    \phi_{*} \circ \varphi_{*}$.
  \end{enumerate}
\end{theorem}

Notice that if $(\Cov{M_{2}}\to M_{2},u_{2})$ is given, and
$M_{1}\xrightarrow{\varphi}M_{2}$ is a map from a smooth closed
manifold $M_{1}$ to $M_{2}$, then $\varphi^{*}(\Cov{M}_{2})$ is an
integration cover of $\varphi^{*}(u_{2})$, but it does not need to
be connected (nor needs $\varphi^{*}(u_{2})$ to be non trivial).
However, we have the following proposition~:
\begin{proposition}\label{prop:funct_pullback}
  In the situation above, each connected component
  ${\varphi^{*}(\Cov{M}_{2})}_{0}$ of $\varphi^{*}(\Cov{M}_{2})$ is an
  integration cover for $\varphi^{*}(u_{2})$. Whenever
  $\varphi^{*}(u_{2})\neq 0$,  the group
  $\pi_{1}({\varphi^{*}(\Cov{M}_{2})}_{0},\varphi^{*}u_{2})$ does not
  depend on the choice of connected component and $\varphi$ induces a
  group morphism
  $$
  \pi_{1}({\varphi^{*}(\Cov{M}_{2})}_{0},\varphi^{*}u_{2})\xrightarrow{\varphi_{*}}
  \pi_{1}(\Cov{M}_{2},u_{2}).
  $$
\end{proposition}

Another feature of this Novikov fundamental group is that it
supports a Hurewicz morphism~:
\begin{theorem}\label{thm:HurewiczMini}
  Suppose  that $\Cov{M}$ is the minimal integration cover.

  Then there is a surjective homomorphism
  $$
  \h: \pi_1(\Cov{M}, u) \to HN_{1}(M,u;\Z).
  $$
  whose kernel is a suitable abelianization $\proabelian{\pi_1(\Cov{M},
    u)}$ that takes the completion process into account (see section
  \ref{sec:hur}).
\end{theorem}

Moreover, the following example shows that this invariant is non
trivial, and different in nature from the invariants derived from
the Novikov homology~:

\begin{theorem}\label{thm:exple}
  Let $S$ be the Poincaré homology sphere and $M=\T^{n}\sharp (S\times
  \S^{n-3})$ be the connected sum of a torus and the product of $S$ with
  a sphere.

  On $M$, consider the class $u=\pi^{*}(d\theta_{1})$ where $\pi$ is the
  projection $M\xrightarrow{\pi_{1}} \S^{1}$ to the first coordinate
  $\theta_{1}$ on the torus.

  Then the Novikov homology associated to the minimal integration cover
  of $u$ vanishes in degrees $1$ and $2$, but the associated Novikov
  fundamental group $\pi_{1}(\Cov{M},u)$ is non trivial. Its minimal
  number of generators and relations up to deck transformations and
  completion are $2$ and $2$.

  In particular, any Morse $1$-form $\alpha$ in the class $u$ necessarily
  has at least $2$ index $1$ and $2$ index $2$ critical points.
\end{theorem}

\begin{remark}\label{rem:ExampleConditions}
  The Poincaré sphere does not play a crucial role here except for the
  explicit bounds $2$ and $2$~: $S\times\S^{n-3}$ could be replaced by
  any closed $n$-manifold $X$, $n\geq 4$, such that
  $H_{1}(M,\Z)=H_{2}(M,\Z)=0$, and the Novikov fundamental group would
  remain non trivial, while the Novikov homology would still vanish.
\end{remark}

Notice that the Novikov homology $HN_{*}(M,u;\pi_{1}(M))$ associated
to the universal cover of $M$ (i.e. when $\Deck=\pi_{1}(M)$)
contains much more information about the fundamental group than that
associated to the minimal integration cover. In particular,
$HN_{1}(M,u;\pi_{1}(M))\neq0$ in this example. However, the second
Novikov homology group $HN_{2}(M,u;\pi_{1}(M))$ still vanishes, so
it does not give any constrain on the number of index $2$ critical
points (while the Novikov fundamental group does).

\begin{remark}
  A nice observation (see beginning of section
  \ref{sec:ComparisonOnUniversalCover}) pointed to us by A.~Pajitnov
  allows to produce index $2$ critical points out the comparison of the
  Novikov homologies associated to the universal and the minimal covers,
  even though $HN_{2}(M,u)=HN_{2}(M,u;\pi_{1}(M))=0$.

  However, a similar example ($\T^{n}\sharp\R\P^{n}$, see
  section \ref{sec:ComparisonOnUniversalCover}) shows that the estimates
  derived from the Novikov fundamental group are also essentially
  different from those obtained in this way.

\end{remark}

\bigskip

The construction of the Novikov fundamental group in this paper is a
natural adaptation to the homotopy setting of the interpretation of
the Novikov complex (and the Novikov Homology itself as well in many
cases, see \cite{Sikorav1987} and \cite{Usher2008}) as the
projective limit of chain complexes relative to sublevels.

A very similar definition appears for higher homotopy groups in
\cite{FGS2010}, but the fundamental group is not discussed, nor a
fortiori the number of generators and relations and their relation
to critical points of Morse $1$-forms in the class $\alpha$.

In \cite{Latour1994}, Latour defines several spaces that are
interesting to compare to the present construction. First, he
considers (\cite[5.7, p. 184]{Latour1994}) the group
$\pi_{1}^{\text{Latour}}(u)$ of loops that can be ``slid to
$-\infty$'' in the minimal integration cover. Despite similar
notations, this group is somewhat orthogonal to our construction
since such loops are automatically trivial in our construction. The
closest notion considered by Latour is the $\pi_{0}$ of the space of
``paths going to $-\infty$''. The Novikov fundamental group we
define here can be thought of as a Novikov completion of the group
generated by this space.

\bigskip

The paper is organized as follows~: the first section is this
introduction, the second section is devoted to the construction of
the Novikov fundamental group and the definition of a suitable
notion of generators and relations, the third to the Morse
interpretation of it and the proof of Theorem \ref{thm:Main}, the
fourth to the discussion of a Novikov version of the Hurewicz
morphism, and finally the last one to the discussion of an example
and the proof theorem \ref{thm:exple}.

\section{Novikov fundamental group}

\subsection{Projective limit with respect to sublevels}

Let $M$ be a closed smooth manifold, $u\in H^{1}(M,\R)$ a non
trivial cohomology class, and $\alpha\in u$ a closed $1$-form in the
class $u$.

Let $p:\Cov{M} \to M$ be an integration cover of $u$ (see Remark
\ref{rk:ChoiceOfCoveringSpace}) and $\Deck =
\pi_{1}(M)/\pi_{1}(\Cov{M})$ the associated deck transformation
group.

Notice that for the minimal integration cover, we have $\pi_1
(\Cov{M}) = \ker u$, and $\Deck = \Z^k$, where $k$ is the
irrationality degree of $u$. In particular, $k =1$ if $u$ is
integral, i.e. if $u \in H^1(M, \Z)$.

The $1$-form $p^*(\alpha)$ on $\Cov{M}$ is exact, and we pick a
primitive $f_{\alpha}:\Cov{M}\to\R$. For $h\in\R$, we let
$$
\sublevel{\Cov{M}}{h} = \{p\in\Cov{M}, f_{\alpha}(p)\leq h\}
$$
$$
\zipped{\Cov{M}}{h} = \Cov{M}/\sublevel{\Cov{M}}{h} = \Cov{M}/\sim
$$
where $\sim$ collapses $\sublevel{\Cov{M}}{h}$ to a point~: $p\sim q
\Leftrightarrow f_{\alpha}(p)\leq h \text{ and } f_{\alpha}(q)\leq
h$. This space comes with a natural base point $\star_{h}$, given by
the collapsed sublevel $\sublevel{\Cov{M}}{h}$.

We now consider the family of groups~:
$$
\zipped{\pi_{1}(f_{\alpha})}{h} = \pi_{1}(\zipped{\Cov{M}}{h},
\star_{h}).
$$
Inclusions of sublevels induce natural maps for any $h,h'\in\R$ with
$h<h'$~:
$$
\zipped{\pi_{1}(f_{\alpha})}{h} \xrightarrow{\zip{h}{h'}}
\zipped{\pi_{1}(f_{\alpha})}{h'}
$$
Moreover, these maps are compatible with successive inclusions~: if
$h,h',h''\in\R$ are such that $h<h'<h''$, then the following diagram
is commutative~:
$$
\xymatrix{%
\zipped{\pi_{1}(f_{\alpha})}{h
}\ar^{\zip{h}{h'}}[r]\ar_{\zip{h}{h''}}@/_{2ex}/[rr]&
\zipped{\pi_{1}(f_{\alpha})}{h'}\ar^{\zip{h'}{h''}}[r] &
\zipped{\pi_{1}(f_{\alpha})}{h''}
}.
$$

As a consequence the projective limit when $h$ goes to $-\infty$ is
well defined~:
\begin{definition}
Define the Novikov fundamental group associated to $\alpha$ as the
projective limit
$$
\pi_{1}(f_{\alpha})=\varprojlim_{h}\zipped{\pi_{1}(f_{\alpha})}{h}.
$$
\end{definition}

As a projective limit, this group comes with maps $\zip{\infty}{h}$
to the ``zipped'' groups~:
$$
\pi_{1}(f_\alpha) \xrightarrow{\zip{\infty}{h}}
\zipped{\pi_{1}(f_\alpha)}{h}.
$$
and each element $g$ has a {\it minimal   height},
$h_{f_{\alpha}}(g)$, with respect to $f_{\alpha}$, defined as
$$
h_{f_{\alpha}}(g)=\inf\{h\in\R, g\in\ker \zip{\infty}{h}\}.
$$

\begin{remark}
  The fact that an element $g$ appears to be trivial above a level
  $a_{0}$ (i.e. $\zip{\infty}{a}(g)=1$) does not mean in general that for
  levels $a<a_{0}$, $g$ can be represented in
  $\zipped{\pi_{1}(f_{\alpha})}{a}$ by a loop that stays below level
  $a_{0}$.

  For instance, if $\gamma$ and $\gamma'$ are two paths running from and
  to $-\infty$, such that for some levels $a_{-}<a_{0}<a_{+}$, we have
  \begin{itemize}
  \item
    $\gamma\neq 1$ in $\zipped{\pi_{1}(f_{\alpha})}{a_{+}}$
  \item
    $\gamma'=1$ in $\zipped{\pi_{1}(f_{\alpha})}{a_{0}}$ but
    $\gamma'\neq1$ in $\zipped{\pi_{1}(f_{\alpha})}{a_{-}}$,
  \end{itemize}
  then the element $g$ defined by $\gamma\gamma'\gamma^{-1}$ is trivial
  in $\zipped{\pi_{1}(f_{\alpha})}{a_{0}}$, but representing its homotopy
  class in $\zipped{\Cov{M}}{a_{-}}$ may require a path going above level
  $a_{+}$ (see the middle picture in Figure \ref{fig:LoopSequence}).
  \end{remark}

\begin{figure}[h!]
  \centering%
  \includegraphics[width=7cm]{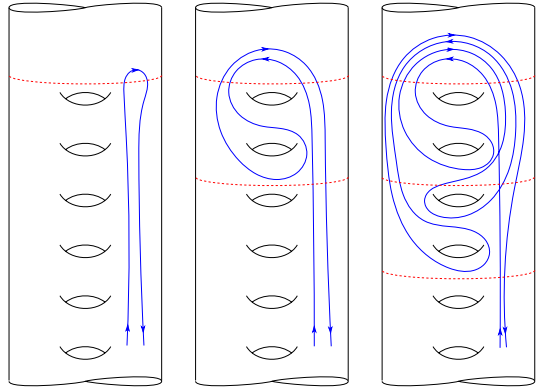}%
  \caption{
    An element in $\pi_1(\Cov{M},u)$ whose deeper and deeper
    representatives do
    not eventually become constant above a given level}%
  \label{fig:LoopSequence}
\end{figure}

In particular, if a path $\R\xrightarrow{\gamma}\Cov{M}$ such that
$\lim_{t\to\pm\infty}f_{\alpha}(\gamma(t))=-\infty$ clearly defines
an element in $\pi_{1}(\Cov{M},u)$, not all the elements can be
described in this way.

Moreover, among such paths, some are better than the others, in that
they do not go down and up too badly~:
\begin{definition}\label{def:CleanElements}
  An element $g$ in $\pi_{1}(f_{\alpha})$ will be said to be clean if it
  can be defined by a path $\R\xrightarrow{\gamma}\Cov{M}$ with
  $\lim_{t\to\pm\infty}f_{\alpha}(\gamma(t))=-\infty$ whose ``wells have
  bounded depth'', i.e. $\exists K>0,
  \forall(t_{1},t_{2})\in\R^{2},\forall h\in\R$
  $$
  f_{\alpha}(\gamma(t_{1}))=f_{\alpha}(\gamma(t_{2}))=h
  \Rightarrow
  \forall t\in[t_{1},t_{2}], f_{\alpha}(t)\geq h-K.
  $$
\end{definition}

\subsection{Invariance}

\begin{theorem}
  The group $\pi_{1}(f_{\alpha})$ does not depend on the choice of the
  $1$-form $\alpha$ nor of the primitive $f_{\alpha}$.

  More precisely, if $\alpha$ and $\beta$ are two forms in the same
  cohomology class $u\in\H^{1}(M,\R)$, $f_{\alpha}$ and $f_{\beta}$ two
   primitives of $\alpha$ and $\beta$ on $\Cov{M}$, then there is a
  canonical isomorphism
  $$
  \pi_{1}(f_{\alpha})\xrightarrow{\sim}\pi_{1}(f_{\beta}).
  $$
\end{theorem}

\begin{definition}
  The identification of all these groups via these canonical isomorphism
  defines a group $\pi_{1}(\Cov{M},u)$ which we call the Novikov
  fundamental group associated to $u$ (and the integration cover
  $\Cov{M}$).
\end{definition}

\begin{remark}\label{rk:Invariance}
  The proof of the invariance will in fact give a little more~: the group
  itself is invariant, but it also comes with a preferred collection of
  projective systems defining it.

  The argument below based on reciprocal inclusions of sublevels
  associated to different forms in the class $u$ will be reused several
  times and will be referred to as the invariance argument in the sequel,
  and we will say that projective systems satisfying the relation
  \eqref{eq:invariance-diagram} for some constant $K$ are ``essentially
  equivalent''.
\end{remark}

\begin{proof}
  Since $\alpha-\beta$ is exact on $M$, there is a function $f$ on $M$
  such that $f_{\beta}=f_{\alpha}+f\circ\pi$, where $\pi:\Cov{M}\to M$ is
  the projection. Since $M$ is compact, $f$ is bounded and there is a
  constant $K$ such that
  $$
  \forall p\in\Cov{M}, f_{\alpha}(p)-K\leq f_{\beta}(p)\leq f_{\alpha}(p)+K.
  $$

  As a consequence, for all $h\in\R$, we have the following inclusions of sublevels
  $$
  \sublevel[f_{\alpha}]{\Cov{M}}{h-K}\subset
  \sublevel[f_{\beta}]{\Cov{M}}{h}\subset
  \sublevel[f_{\alpha}]{\Cov{M}}{h+K}
  $$
  which induce morphisms among relative fundamental groups, and make the
  following diagram commutative~:
  \begin{equation}
    \label{eq:invariance-diagram}
    \xymatrix{ \dots\ar[r] &
      \zipped{\pi_{1}(f_{\alpha})}{h-K}     \ar[dr]\ar[rr]
               && \zipped{\pi_{1}(f_{\alpha})}{h+K}     \ar[dr]\ar[r] & \dots\\
    \dots\ar[rr]\ar[ur]
               && \zipped{\pi_{1}(f_{\beta })}{h  }     \ar[ur]\ar[rr]
               && \dots\\
  }
  \end{equation}
  which in turn induce an isomorphism on the projective limit of the
  fundamental groups.
\end{proof}

\subsection{Deck transformations}

For any $x \in \Cov{M}$  and any $\tau \in \Deck$ we have
\begin{equation*}
f_\alpha( \tau \circ x) = f_\alpha(x) + u(\tau),
\end{equation*}
so that deck transformations send sublevels to sublevels~:
$$
\tau\cdot\sublevel{\tilde {M}} {h} = \sublevel{\Cov{M}}{h+u(\tau)}.
$$
Hence $\tau$ also acts on $\pi_{1}(f_\alpha)$. As a consequence, for
all $h\in\R$, we have an isomorphism
$$
\zipped{\pi_{1}(f_{\alpha})}{h}\xrightarrow{\tau}
\zipped{\pi_{1}(f_{\alpha})}{h+u(\tau)}.
$$
This induces an isomorphism on the projective limit
$$
\pi_{1}(f_{\alpha})\xrightarrow{\tau} \pi_{1}(f_{\alpha})
$$
which finally defines an action of $\Deck$ on $\pi_{1}(f_{\alpha})$.

\subsection{Functoriality of the Novikov  fundamental group}\label{sec:funct}
Here we prove theorem \ref{thm:funct}.

We start with an arrow $(\Cov{M_{1}} \to M_{1},u_{1})
\xrightarrow{\varphi} (\Cov{M_{2}} \to M_{2},u_{2})$ i.e. with a
commutative diagram
$$
\xymatrix{
  \Cov{M_{1}}\ar_{\pi_{1}}[d]\ar^{\tilde{\varphi}}[r]&
  \Cov{M_{2}}\ar^{\pi_{2}}[d]
  \\
  M_{1}\ar^{\varphi}[r]&
  M_{2}
}
$$
such that $\varphi^{*}(u_{2})=u_{1}$.

Consider a $1$-form $\alpha_{2}$ in the class $u_{2}$ and a
primitive $f_{2}$ of $\pi_{2}^{*}\alpha_{2}$ on $\Cov{M_{2}}$.

Then $f_{1} = \tilde{\varphi}^{*}f_{2}$ satisfies
$$
df_{1} = d(\tilde{\varphi}^{*} f_{2}) = \tilde{\varphi}^{*}df_{2} =
\tilde{\varphi}^{*}\pi_{2}^{*}\alpha_{2} =
\pi_{1}^{*}\varphi^{*}\alpha_{2},
$$
and hence is a primitive of $\pi_{1}^{*}\alpha_{1}$, where
$\alpha_{1} = \varphi^{*}\alpha_{2}$ is a one form in the cohomology
class $u_{1}$. Moreover, the sublevels associated to $f_{1}$ are the
preimages of those associated to $f_{2}$. In particular,
\begin{equation*}
  \tilde{\varphi}(\{f_{1}\leq h\}) \subset \{f_{2}\leq h\},
\end{equation*}
and $\varphi$  induces  a homomorphism
\begin{equation*}
  \zipped{\varphi_{*}}{h}: \zipped{\pi_1 (f_{1})}{h} \to
  \zipped{\pi_1(f_{2})}{h}.
\end{equation*}
The collection $(\zipped{\varphi_{*}}{h})_{h\in\R}$ forms an inverse
system morphism, and in the limit, we obtain a group morphism~:
\begin{equation}
  \label{eq:phi*Novikov}%
  \varphi_{*} : \pi_1 (\Cov{M}_{1},u_{1}) \to
  \pi_1 (\Cov{M}_{2}, u_{2}).
\end{equation}

Expressing the compatibility of $\varphi_{*}$ with deck
transformations requires some more discussion of the latter.

Let $\Deck_{\Cov{M}_{i}}$ be the deck transformation group of
$\Cov{M}_{i}$. Observe that since $\Cov{M}_{2}$ is connected, there
a is natural morphism
$$
\Deck_{\Cov{M}_{1}}\xrightarrow{\varphi_{\star}}\Deck_{\Cov{M}_{2}}.
$$
Indeed, consider a point $p\in\Cov{M_{1}}$ and a loop
$\gamma:[0,1]\to M_{1}$ based at $\pi(p)$, representing a homotopy
class $g_{1}\in\pi_{1}(M_{1})$. Let $\tilde{\gamma}$ be the lift of
$\gamma$ at $p$, and $\widetilde{\varphi(\gamma)}$ the lift of
$\varphi_{*}\gamma$ at $\tilde\varphi(p)$. Then $\tilde{\gamma}(1) =
g\cdot p$ and $\tilde{\varphi}(\tilde{\gamma}(1)) =
(\widetilde{\varphi_{*}\gamma})(1)$, which means that
\begin{equation}
  \label{eq:phideck}
\tilde\varphi(g\cdot p) = \varphi_{*}(g)\cdot \tilde\varphi(p).
\end{equation}
In particular, if $g$ and $g'$ in $\pi_{1}(M_{1})$, have the same
action on $\Cov{M_{1}}$, then $\varphi_{*}(g)$ and $\varphi_{*}(g')$
have the same action at least on a point in $\Cov{M_{2}}$, and since
$\Cov{M_{2}}$ is connected, they have the same action on the whole
of $\Cov{M}_{2}$.

This means that
$\pi_{1}(\Cov{M_{1}})\xrightarrow{\varphi_{*}}\pi_{1}(\Cov{M}_{2})$
induces a map
$\Deck_{\Cov{M_{1}}}\xrightarrow{\varphi_{*}}\Deck_{\Cov{M}_{2}}$
for which we have~:
$$
\forall g\in\Deck_{\Cov{M}_{1}},\ \tilde{\varphi}\circ g =
\varphi_{*}(g)\circ \tilde{\varphi}.
$$

With this equality we have the following commutative diagram
$$
\xymatrix{
  \sublevel{\Cov{M_{1}}} {h} \ar_{g}[d]\ar^{\tilde{\varphi}}[r]&
  \sublevel{\Cov{M_{2}}} {h}\ar^{\varphi_{*}(g)}[d]
  \\
  \sublevel{\Cov{M_{1}}} {h+u_1(g)}\ar^{\tilde{\varphi}}[r]&
  \sublevel{\Cov{M_{2}}} {h+u_1(g)}
}
$$
with $\sublevel{\Cov{M_{2}}} {h+u_1(g)}=\sublevel{\Cov{M_{2}}}
{h+u_2(\varphi_{*}(g))}$, that induces the commutative diagram on
the Novikov $\pi_1$.

This ends the proof of point (1) of theorem \ref{thm:funct}. Point
(2) follows along the same lines and is no more difficult.

\subsection{Generators and relations}

\subsubsection{Generators up to deck transformations and completion}
In general, $\pi_{1}(\Cov{M},u)$ need not be finitely generated in
the usual sense. The object of this section is to define a suitable
notion of generators that takes both deck transformations and
projective limits into account, for which $\pi_{1}(\Cov{M},u)$ will
be finitely generated.

Given a subset $A\subset\pi_{1}(\Cov{M},u)=\pi_{1}(f_{\alpha})$,
consider its orbit $\Deck\cdot A=\{\tau\cdot g, \tau\in\Deck, g\in
A\}$ under all possible deck transformations, and the subgroup
$\PreSpan{A}$ generated by $\Deck\cdot A$ in $\pi_{1}(f_{\alpha})$.
For each $h\in\R$, the image of this subgroup in
$\zipped{\pi_{1}(f_{\alpha})}{h}$ defines a group
$$
\zipped{\PreSpan{A}}{h}=\zip{-}{h}(\PreSpan{A})\subset\zipped{\pi_{1}(f_{\alpha})}{h}
$$
making the following diagrams commutative
$$
\xymatrix{%
\zipped{\PreSpan{A}}{h  }\ar[r]^{\zip{h
}{h'}}\ar@/_{2ex}/[rr]_{\zip{h}{h''}}& \zipped{\PreSpan{A}}{h'
}\ar[r]^{\zip{h'}{h''}} & \zipped{\PreSpan{A}}{h''}
}.
$$

\begin{definition}
  Define the subgroup generated by $A$ up to deck transformations and completion as
  the group
  $$
  \Span{A} =
  \varprojlim_{h}\zipped{\PreSpan{A}}{h}.
  $$
\end{definition}

\begin{remark}
  Notice that $\Span{A}$ is bigger in general than the subgroup
  $\PreSpan{A}$, since the latter involves only (arbitrary long but)
  finite products of elements of $\Deck\cdot A$, while the limit process
  allows for infinite products.

  In the sequel, generated subgroups will always be understood as
  generated up to deck transforms and completion.
\end{remark}

\begin{proposition}
  The group $\Span{A}$ generated by $A$ in the sense above only depends on the
  class $u$, and not on the choice of $\alpha$ nor $f_{\alpha}$ used to
  define it.
\end{proposition}

\begin{proof}
Suppose $\alpha$ and $\beta$ are two 1-forms in the class $u$,
$f_{\alpha}$ and $f_{\beta}$ two primitives of $\alpha$ and $\beta$
on $\tilde{M}$.

Let $K=\Vert f_{\alpha}-f_{\beta}\Vert_{\infty}$. Then the inclusion
of sublevels induces, for any level $h\in\R$, the following
commutative diagram
$$
  \xymatrix@!C=5em{
    \dots\ar[r] & \zipped{\PreSpan{A}}{f_{\alpha}\leq h-K}     \ar[dr]\ar[rr]
               && \zipped{\PreSpan{A}}{f_{\alpha}\leq h+K}     \ar[dr]\ar[r] & \dots\\
    \dots\ar[rr]\ar[ur]
               && \zipped{\PreSpan{A}}{f_{\beta }\leq h  }     \ar[ur]\ar[rr]
               && \dots\\
  }
$$
 which again induces an isomorphism of projective limits.
\end{proof}

\begin{remark}
  The invariance principle used in this proof shows that the notion of
  ``clean'' elements defined in \ref{def:CleanElements} does not depend
  on the choice of $\alpha$ nor $f_{\alpha}$.
\end{remark}

\subsubsection{Relations  up to deck transformations and completion}
\label{sec:Relations}%
To define a notion of relation, we need a notion of free group
generated by some elements, that still takes the deck
transformations and the completion process into account.

Given a finite set $A=\{g_{1},\dots, g_{k}\}\subset
\pi_{1}(f_{\alpha})$, we consider the product $\Deck\times A$ and
denote by $\PreFree{A}$ the group freely generated by its elements.
Notice it supports an action of the deck transformations group, by
multiplication of the letters in a word~:
\begin{gather*}
  \tau\cdot(w_{1}\dots w_{k})= (\tau\cdot w_{1})\dots (\tau\cdot w_{k})
  \intertext{ and for each letter $w_{i}=(\tau_{i},g_{i})$}
  \tau\cdot(\tau_{i},g_{i}) = (\tau\tau_{i},g_{i}).
\end{gather*}

The height
\begin{equation}
  \label{eq:height}%
  h_{f_{\alpha}}(g)=\inf\{h\in\R,
  \zip{\infty}{h}(g)=1\in\zipped{\pi_{1}(f_{\alpha})}{h}\}
\end{equation}
of elements in $\pi_{1}(f_{\alpha})$ extends to $\PreFree{A}$ at the
level of letters by letting
$$
h_{f_{\alpha}}(\ (\tau,g)\ )=u(\tau)+h_{f_{\alpha}}(g) \quad\text{
and }\quad h_{f_{\alpha}}(\ w^{-1}\ )=h_{f_{\alpha}}(w)
$$
and at the level of words by letting
$$
h_{f_{\alpha}}(w_{1}\dots w_{k})=\sup_{1\leq i\leq
k}\{h_{f_{\alpha}}(w_{i})\}.
$$

\begin{remark}\label{rk:ArbitraryHeight}
  In fact, as long as there are finitely many letters, the height of the
  letters could be fixed arbitrarily.
\end{remark}

Now, for each level $h$ we select the subset
$$
\PreFree{A}^{\leq h}= \{w\in \PreFree{A}, h_{f_{\alpha}}(w)\leq h\},
$$
and consider  the group $\zipped{\PreFree{A}}{h} =
\PreFree{A}/\PreFree{A}^{\leq h}$.

Observe now that for $h<h'<h''$, we have compatible morphisms, still
denoted by $\zip{}{}$~:
$$
\xymatrix{%
\zipped{\PreFree{A}}{h
}\ar^-{\zip{h}{h'}}[r]\ar_-{\zip{h}{h''}}@/_{2ex}/[rr]&
\zipped{\PreFree{A}}{h' }\ar^-{\zip{h'}{h''}}[r] &
\zipped{\PreFree{A}}{h''}
}
$$
which map all the letters whose height is too small to $1$.

\begin{definition}
Define the group freely generated by $A$ up to deck transformations
and completion by
$$
\Free{A}=\varprojlim_{h} \zipped{\PreFree{A}}{h}.
$$
\end{definition}

The evaluation map sending each letter $g_{i}$ to the element
$g_{i}\in\pi_{1}(f_{\alpha})$ induces morphisms
\begin{equation}
  \label{eq:eval_h}
  \begin{array}{ccc}
    \zipped{\PreFree{A}}{h}&\xrightarrow{\eval_{h}}&\zipped{\pi_{1}(f_{\alpha})}{h+K}\\
    \prod_{j=1}^{N}(\tau_{j},g_{i_{j}})&\mapsto& %
    \zip{\infty}{h+K}(\prod_{j=1}^{N}\tau_{j}\cdot g_{i_{j}} )
\end{array}
\end{equation}
where $K$ is a constant, that can be taken to be $1$ when the
convention \eqref{eq:height} is used to define the height, but more
generally satisfies~:
$$
\forall i\in\{1,\dots,m\}: h_{f_{\alpha}}(g_{i})=h \Rightarrow
\zip{\infty}{h+K}(g_{i})=1.
$$
Finally, the applications $\eval_{h}$ induce a morphism~:
\begin{equation}\label{eq:evalFreeToPi1}
  \xymatrix{ \Free{A} \ar^-{\eval}@{>>}[r]& \Span{A}\subset
    \pi_{1}(f_{\alpha}). }
\end{equation}

\begin{definition}
  Define the group $\Rel(A)$ of relations associated to a generating
  family $A$ of $\pi_{1}(f_{\alpha})$ up to DTC as the kernel of
  $$
  \Free{A} \xrightarrow{\eval} \pi_{1}(f_{\alpha}).
  $$
\end{definition}

\begin{remark}
  This kernel is compatible with the filtration, i.e.
  $$
  \forall w\in\Free{A}, w\in\ker \eval
  \Leftrightarrow
  \forall h\in\Z, \zip{\infty}{h}(w)\in\ker\eval_{h} .
  $$
\end{remark}

The relations are not finitely generated in general, but there is
again a notion of subgroup ``normally'' generated up to DTC by a
collection of elements~: given a set $B\subset\Rel(A)$, define
$B'=\{g b g^{-1}, g\in \Free{A}, b\in B\}$  and let
$$
\NormalSpan{B}=\Span{B'}
$$
be the subgroup generated up to DTC by all the conjugates of
elements of $B$.

Given a set $A\subset\pi_{1}(f_{\alpha})$ of generators of
$\pi_{1}(f_{\alpha})$ up to DTC, we let
$$
\minrel(A)= \inf\{\sharp B, B\subset \Rel(A) \text{ with } \Rel(A) =
\NormalSpan{B}\} \in\bar{\R}
$$

\begin{definition}
  Let
  $$
  \mingen(\pi_{1}(f_{\alpha})) = \inf\{\sharp A, A\subset \pi_{1}(f_{\alpha}),
  \pi_{1}(f_{\alpha}) \text{ is generated by } A \text{ up to DTC}\}
  $$
  and
  $$
  \minrel(\pi_{1}(f_{\alpha})) = \inf\{\minrel(A), A\subset \pi_{1}(f_{\alpha}),
  \pi_{1}(f_{\alpha}) \text{ is generated by } A \text{ up to DTC}\}
  $$
  denote the minimal number of generators and relations (respectively) of
  $\pi_{1}(f_{\alpha})$.
\end{definition}

\begin{remark}
  The invariance argument (see remark \ref{rk:Invariance}) shows that the
  numbers $\mingen(\pi_{1}(f_{\alpha}))$ and
  $\minrel(\pi_{1}(f_{\alpha}))$ do not depend on the choices of $\alpha$
  and $f_{\alpha}$, but only on $\pi_{1}(\Cov{M},u)$.
\end{remark}

The object of the next section is to prove the following~:
\begin{theorem}\label{thm:main}
  The group $\pi_{1}(\Cov{M}, u)$ is finitely generated up to deck
  transformations and completion, and for all Morse $1$-form $\alpha$ in
  the cohomology class $u$, we have
  $$
  \sharp(\Crit_{1}(\alpha))\geq \mingen(\pi_{1}(f_{\alpha})).
  $$
  Moreover, the relations are also finitely generated up to deck
  transformations and completion, and
  $$
  \sharp(\Crit_{2}(\alpha))\geq \minrel(\pi_{1}(f_{\alpha})).
  $$
\end{theorem}

\begin{remark}\label{rem:sharko}
   A Morse version  of Theorem \ref{thm:main}  is   stated in
   \cite[Theorem 7.10]{Sharko1993}.
\end{remark}

\begin{remark}
  The proof provides a slightly sharper result, where only ``clean''
  generators (see definition \ref{def:CleanElements}) are taken into
  account.
\end{remark}

\section{Morse theoretic interpretation}

\subsection{Morse-Novikov steps as generators}%
Suppose now that $\alpha$ is a Morse $1$-form in the cohomology
class $u\in H^{1}(M,\R)$ (with $u\neq 0$).

Pick also
\begin{itemize}
\item
  a metric $<,>$ on $M$ such that the pair $(\alpha,<,>)$ is Morse-Smale,
\item
  a primitive $f_{\alpha}$ of $\alpha$ on $\Cov{M}$,
\item
  a preferred lift $\tilde{\underbar{c}}$ to $\Cov{M}$ of each
  $\underbar{c} \in\Crit(\alpha)$~; this allows for the identification
  $$
  \Crit(f_{\alpha})=\{\tau\cdot\tilde{\underbar{c}},\
  \tau\in\Deck,\underbar{c}\in\Crit(\alpha)\}.
  $$
\item
  an arbitrary orientation on the unstable manifold of each
  $c\in\Crit(\alpha)$~; this picks a preferred orientation on the
  unstable manifolds of all the critical points of $f_{\alpha}$.
\end{itemize}

For convenience, we suppose $u\neq 0$. This allows us to pick, for
each index $0$ critical point $\underbar{x}\in\Crit_{0}(\alpha)$, an
arbitrary path $\gamma_{\tilde{\underbar{x}}}:[0,+\infty)\to\Cov{M}$
such that
\begin{itemize}
\item
  $\gamma_{\tilde{\underbar{x}}}(0)=\tilde{\underbar{x}}$
\item
  $\lim_{t\to+\infty}f_{\alpha}(\gamma_{\tilde{\underbar{x}}}(t))=-\infty$.
\end{itemize}
In fact, we can pick a loop $l$ in $M$ based at $\underbar{x}$ such
that $u(l)<0$, and lift the iterations of $l$ to $\Cov{M}$.

Notice that, letting
$$
\gamma_{\tau\cdot\tilde{\underbar{x}}} = \tau\cdot
\gamma_{\tilde{\underbar{x}}},
$$
automatically selects, for each index $0$ critical point $x$ of
$f_{\alpha}$, a preferred   path   $\gamma_{x}$ in $\Cov{M}$ from
$x$ to
$-\infty$.%

There are finitely many index $0$ critical points for $\alpha$, so
there is an upper bound $\kappa$ on the height these paths can reach
above their starting point~:
\begin{equation}\label{eq:gammaxHeightBound}
  \exists\kappa\in\R, \forall x\in\Crit_{0}(f_{\alpha}),\quad
  \gamma_{x}\subset\sublevel{\Cov{M}}{f_{\alpha}(x)+\kappa}
\end{equation}

\bigskip

To each index $1$ critical point $y$ of $f_{\alpha}$, the unstable
manifold of $y$, endowed with its preferred orientation, defines a
path $\gamma:(-\infty,+\infty)\to\Cov{M}$ such that near both ends,
we have the alternative
\begin{gather*}
\lim_{t\to -\infty}f_{\alpha}(\gamma(t))=-\infty \text{ or }
\lim_{t\to-\infty}f_{\alpha}(\gamma(t))=x_{-}\in\Crit_{0}(f_{\alpha})\\
\intertext{and} \lim_{t\to+\infty}f_{\alpha}(\gamma(t))=-\infty
\text{
  or }
\lim_{t\to+\infty}f_{\alpha}(\gamma(t))=x_{+}\in\Crit_{0}(f_{\alpha}).
\end{gather*}

In the former case, the corresponding flow line is said to be
infinite, or unbounded. In the latter case, it is said to be
bounded, and a suitable re-parameterisation (using the value of
$f_{\alpha}$ as the new parameter for instance) and concatenation
with the path $\gamma_{x_{\pm}}$ extends the flow line into a path
going to $-\infty$. Doing this continuation at one or both ends if
needed, we obtain in all cases a continuous path
\begin{equation*}
  \gamma_{y}:(-\infty,+\infty)\to\Cov{M}  \text{ with  }
  \lim_{t\to\pm\infty}f_{\alpha}(\gamma_{y}(t))=-\infty,
\end{equation*}
which we call {\it the Morse-Novikov step} associated to $y$.
\begin{figure}[h!]
  \centering \includegraphics[scale=.6]{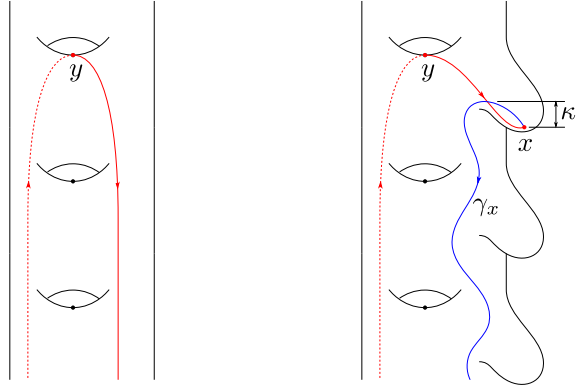}
  \caption{Morse-Novikov steps. Whenever a flow line rooted at $y$ ends
    at an index $0$ critical point $x$, it is extended with the path
    $\gamma_{x}$.} \label{fig:MorseNovikovSteps}
\end{figure}

\begin{definition}
Let $\Steps$ be the collection of Morse-Novikov steps associated to
the index $1$ critical points of $\alpha$~:
$$
\Steps = \{\gamma_{\tilde{\underbar{y}}},
\underbar{y}\in\Crit_{1}(\alpha)\}.
$$
Each such path $\gamma_{\tilde{\underbar{y}}}$ defines a class
$g_{\tilde{\underbar{y}}} = [\gamma_{\tilde{\underbar{y}}}] \in
\pi_{1}(f_{\alpha})$, and $\Steps$ will often be implicitly and
abusively considered as a finite subset in $\pi_{1}(f_{\alpha})$, by
considering the collection $\{g_{\tilde{\underbar{y}}}\}$ instead of
$\{\gamma_{\tilde{\underbar{y}}}\}$.
\end{definition}

\begin{remark}
  The Morse-Novikov steps are ``clean'' elements in the sense of
  definition \ref{def:CleanElements}.
\end{remark}

\begin{proposition}\label{prop:step}
The collection $\Steps$ is a finitely generating family (up to deck
transformations and completion) for $\pi_{1}(\Cov{M},u)$~:
$$
\pi_{1}(\Cov{M},u) = \Span{\Steps}= \Span{\{g_{\tilde{\underbar y}},
\underbar{y}\in\Crit_{1}(f_{\alpha})\}}.
$$
\end{proposition}

A straightforward corollary of this proposition is the following~:
\begin{corollary}
  For all Morse $1$-form $\alpha$ in the class $u$ we have
  \begin{equation}\label{eq:Crit1LowerBound}
    \sharp\Crit_{1}(\alpha)\geq \mingen(\pi_{1}(\Cov{M},u)).
  \end{equation}
\end{corollary}

\begin{figure}[h!]
  \centering \includegraphics[scale=.9]{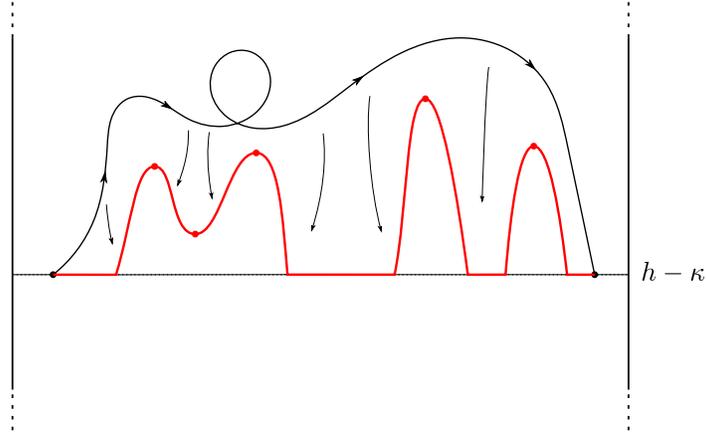} \caption{Starting
    from an arbitrary loop, flow it down...} \label{fig:HomotopyFlow}
\end{figure}
\begin{figure}[h!]
  \centering \includegraphics[scale=1.1]{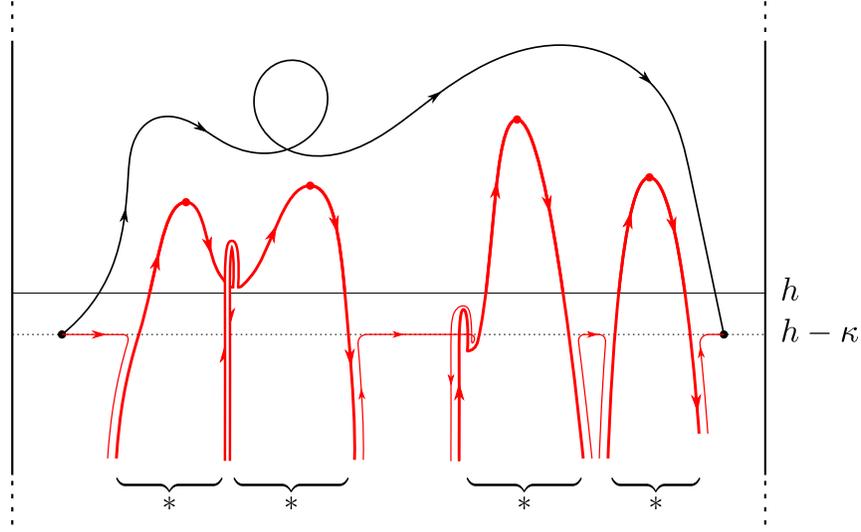} \caption{...and
    complete it to get a sequence of Morse-Novikov steps. The paths
    marked with a $(\ast)$ are Morse-Novikov steps, the other are
    contained in $\sublevel{\Cov{M}}{h}$.} \label{fig:HomotopyFinal}
\end{figure}

\begin{proof}
Fix some regular level $h$ of $f_{\alpha}$. Let
$g\in\pi_{1}(f_{\alpha})$, and
$\gamma:[0,1]\to\zipped{\Cov{M}}{h-\kappa}$ be a path representing
$\zip{}{h-\kappa}(g)$ in $\zipped{\pi_{1}(f_{\alpha})}{h-\kappa}$
(recall $\kappa$ is the constant defined in
\eqref{eq:gammaxHeightBound}).

The path $\gamma$ defines a finite collection of paths
$(\gamma_{1},\dots,\gamma_{k})$ from $[0,1]$ to $\Cov{M}$ with ends
on $\{f_{\alpha}=h-\kappa\}$.

Pushing a component $\gamma_{i}$ down by the gradient flow of
$-f_{\alpha}$ moves it either below level $h$ or onto unstable
manifolds of index $1$ critical points. More precisely, there is a
time $T$ after which the curve is contained in the union of the
sublevel $\sublevel{\Cov{M}}{h}$ and small neighborhoods of the
unstable manifolds of some index $1$ critical points. A classical
retraction argument then turns it into the concatenation
$\gamma'_{i}$ of paths that are (see figure \ref{fig:HomotopyFlow}).
\begin{enumerate}
\item
  either a path in $\sublevel{\Cov{M}}{h}$,
\item
  or a piece of unstable manifold of $y\in\Crit_{1}(f_{\alpha})$ whose
  ends are either in $\sublevel{\Cov{M}}{h}$ or some index $0$ critical
  point $x\in\Crit_{0}(f_{\alpha})$ with $f_{\alpha}(x)>h-\kappa$.
\end{enumerate}

Notice that since the initial path $\gamma_{i}$ starts and ends in
$\sublevel{\Cov{M}}{h}$, so does $\gamma'_{i}$.

To turn $\gamma'_{i}$ into a product of Morse-Novikov steps, we
finally apply the following modifications at the bottom of each
piece of flow line used in $\gamma'_{i}$~:
\begin{enumerate}
\item
  if the flow line is unbounded, we insert a two way trip down to $-\infty$
  along the flow line,
\item
  if the flow line is bounded, we insert a two way trip along the
  remaining part of the flow line if required, and along the preferred
  path $\gamma_{x}$ associated to this critical point down to $-\infty$.
\end{enumerate}

This operation completes each arc of unstable manifold in
$\gamma'_{i}$ into the associated Morse-Novikov step.

Each piece $\eta$ of $\gamma'_{i}$ that is contained in
$\sublevel{\Cov{M}}{h-\kappa}$ is now completed
\begin{itemize}
\item
  either by a flow line that goes to $-\infty$, which is necessarily
  also contained in $\sublevel{\Cov{M}}{h-\kappa}$,
\item
  or a piece of flow line down to some critical point
  $x\in\Crit_{0}(f_{\alpha})$, followed by the path $\gamma_{x}$. Since
  $x$ has to be on a lower level than the start or end of $\eta$, we have
  $f_{\alpha}(x)\leq h-\kappa$, and then by definition of $\kappa$,
  $\gamma_{x}$ cannot go higher than the level $(h-\kappa)+\kappa=h$.
\end{itemize}

As a consequence we obtain a sequence of paths from and to
$-\infty$, such that
\begin{enumerate}
\item
  each path is either a Morse-Novikov step or lies in
  $\sublevel{\Cov{M}}{h}$,
\item
  in $\zipped{\Cov{M}}{h}$, the concatenation of these paths is homotopic
  to $\gamma_{i}$.
\end{enumerate}

Finally, $\gamma'_{i}$ is homotopic in $\zipped{\Cov{M}}{h}$ the
projection of a product of Morse-Novikov steps~:
$$
\gamma'_{i} \sim \gamma_{y_{1}}^{\pm1}\cdots\gamma_{y_{k}}^{\pm1}
\quad\text{ in }\quad \zipped{\Cov{M}}{h}
$$

In particular, we obtain that the class $g$ we started with
satisfies
$$
\zip{}{h}(g)\in\zipped{\PreSpan{\Steps}}{h}.
$$
This proves that $\forall h\in\R$,
$\zipped{\pi_{1}(f_{\alpha})}{h}\subset\zipped{\PreSpan{\Steps}}{h}$
and hence
$$
\pi_{1}(f_{\alpha}) = \Span{\Steps}.
$$
\end{proof}

\subsection{Morse-Novikov relations }

Similarly, consider an index $2$ critical point
$z\in\Crit_{2}(f_{\alpha})$.

Given a level $h$ with $h<f_{\alpha}(z)$, the set
$W^{u}(z)\cap\{f_{\alpha}> h\}$ is a topological disc (notice that
once a trajectory enters the sublevel $\{f_{\alpha}\leq h\}$, it
will never exit it anymore). Moreover, starting with a small circle
around $z$ inside its unstable manifold, and pushing it down by the
flow using the technique described earlier, we obtain a loop
$\rho_{z,h}$ whose projection in $\zipped{\Cov{M}}{h}$ is that of a
product of Morse-Novikov steps~:
$$
\rho_{z,h} \sim \gamma_{y_{1}}^{\pm1}\cdots\gamma_{y_{k}}^{\pm1}
\quad\text{ in }\quad \zipped{\Cov{M}}{h}.
$$
In particular, considering the group $F=\Free{\Crit_{1}(\alpha)}$
freely generated by the index $1$ critical points of $\alpha$ (up to
deck transformations and completion), the sequence
$(y_{1}^{\pm1},\dots,y_{k}^{\pm1})$, after removal of the eventual
critical points $y$ such that $h-\kappa<f_{\alpha}(y)\leq h$,
defines a word
$$
w_{z,h}=y_{1}^{\pm1}\dots y_{k}^{\pm1}\in\zipped{F}{h}.
$$

Moreover, these words are compatible with inclusion of sublevels,
namely, for $h<h'$, we have~:
$$
\zip{h}{h'}(w_{z,h}) = w_{z,h'}.
$$
As a consequence, the words $(w_{z,h})_{h}$ define a class $w_{z}$
in $F$.

\begin{definition}
  Define the relation associated to a critical point
  $\underbar{z}\in\Crit_{2}(\alpha)$ as the word
  $w_{\tilde{\underbar{z}}}$, and let
  $$
  \Rels[\alpha] = \{w_{\tilde{\underbar{z}}}, \underbar{z}\in\Crit_{2}(\alpha)\}.
  $$
\end{definition}

\begin{proposition}\label{prop:Crit2SpansRelations}
  The relations associated to the generating family given by the index
  $1$ critical points of $\alpha$ is normally spanned by the relations
  associated to the index $2$ critical points~:
  $$
   \Rel(\Steps) = \NormalSpan{\Rels}.
  $$
\end{proposition}

A straightforward corollary of this proposition is the following~:
\begin{corollary}
  For all Morse $1$-form $\alpha$ in the class $u$ we have
  \begin{equation}\label{eq:Crit2LowerBound}
    \sharp\Crit_{2}(\alpha)\geq \minrel(\pi_{1}(\Cov{M},u)).
  \end{equation}
\end{corollary}

\begin{proof}
  By construction, the elements in $\NormalSpan{\Rels}$ are indeed
  relations. On the other hand, let $w\in F=\Free{\Steps}$ and suppose
  $w$ evaluates to $1$ in all $\zipped{\pi_{1}(f_{\alpha})}{h}$ (via the
  map defined in \eqref{eq:evalFreeToPi1}).

  Fix a level $h$, and consider $w_{h}=\zip{\infty}{h}(w)\in\zipped{F}{h}$ and
  $\gamma_{h}$ its evaluation in $\zipped{\Cov{M}}{h}$. Since
  $[\gamma_{h}]=1$ in $\zipped{\pi_{1}(f_{\alpha})}{h}$, there is a disc
  $\delta:D^{2}\to\zipped{\Cov{M}}{h}$ with boundary on $\gamma_{h}$.

  Pushing this disc down by the flow, we obtain a new disc
  $\delta'~:D^{2}\to\zipped{\Cov{M}}{h}$, that has the same boundary
  (since it did already consist of flow lines) and splits as a union of
  unstable manifolds of index $2$ critical points (seen in
  $\zipped{\Cov{M}}{h}$), and regions where it evaluates to the base point
  $[\sublevel{\Cov{M}}{h}]$.

  In particular, we obtain that $\gamma_{h}$ can be written in
  $\zipped{\Cov{M}}{h}$ as a product of conjugates of the relations
  associated to index $2$ critical points.
\end{proof}

\section{Hurewicz  homomorphism}\label{sec:hur}

In this section we construct a Hurewicz morphism in the Novikov
setting.

We refer to \cite{Sikorav1987, Farber2004} for a definition of the
Novikov homology $HN_*(\Cov{M}, u)$ associated to a cohomology class
$u\in H^1(M, \R)$ and an associated integration cover $\tilde M$.
The Novikov homology associated to the minimal integration cover
will be denoted as $HN_{*}(M,u)$, and the one associated to the
universal cover by $HN_{*}(M,u;\Z[\pi_{1}(M)])$.

\subsection{Commutators and projective limits}
To describe the kernel of the Hurewicz morphism in the Novikov
setting, we first need to discuss commutators in relation to
projective limits. Let $(G_{h})_{h}$ be a projective system of
groups and $G=\varprojlim G_{h}$. Then both the commutators
$([G_{h},G_{h}])_{h}$ and the abelianizations
$(\abelian{G_{h}})_{h}=(G_{h}/[G_{h},G_{h}])_{h}$ form two other
projective systems. The projective limit of the commutator subgroup
$\varprojlim[G_{h},G_{h}]$ form a normal subgroup of $G$ that
contains the usual commutator subgroup $[G,G]$ of $G$. We refer to
it as the pro-commutator subgroup.

Recall that a projective system $(G_{h})_{h}$ with maps
$G_{h}\xrightarrow{\zip{h}{h'}} G_{h'}$ is said to satisfy the
Mittag-Leffler condition if there is a constant $K$ such that, for
each fixed level $h_{0}$, the maps $\zip{h}{h_{0}}: G_{h}\to
G_{h_{0}}$ have the same range in $G_{h_{0}}$ for all $h\leq
h_{0}-K$.

\begin{proposition}
  If the projective system $(G_{h})_{h}$ satisfies the Mittag-Leffler
  condition, then
  $$
  G/\varprojlim[G_{h},G_{h}] \simeq \varprojlim
  (G_{h}/[G_{h},G_{h}]).
  $$
  This group will be denoted as $\proabelian{G}$ and referred to as the
  pro-abelianization of $G$ (although it depends on the projective system
  $(G_{h})_{h}$ and not on $G$ only).
\end{proposition}

Notice that since $[G,G]$ is a subgroup of $\varprojlim
[G_{h},G_{h}]$, there is a projection
$\abelian{G}\to\proabelian{G}$, but these groups are different in
general.

\begin{proof}
  At each level, we have the following exact sequence
  $$
  0\to [G_{h},G_{h}]\to G_{h} \to G_{h}/[G_{h},G_{h}]\to 0,
  $$
  which form in fact an exact sequence of projective systems.

  Observe now that if $(G_{h})_{h}$ satisfies the Mittag-Leffler
  condition, so does $([G_{h},G_{h}])_{h}$. As a consequence, we have an
  exact sequence among the projective limits~:
  $$
  0\to
  \varprojlim [G_{h},G_{h}]\to
  G \to
  \varprojlim (G_{h}/[G_{h},G_{h}])\to 0,
  $$
  which shows that $G/\varprojlim[G_{h},G_{h}]$ is isomorphic to
  $\varprojlim (G_{h}/[G_{h},G_{h}])$.
\end{proof}

\begin{lemma}
  The Novikov fundamental group satisfies the Mittag-Leffler condition,
  namely, if $\alpha$ is a $1$-form in the class $u$, and $f_{\alpha}$ a
  primitive of $\alpha$ on $\Cov{M}$, then there exist
  a constant $\kappa$ such that, for all level $h\in\R$ the maps
  $$
  \zipped{\pi_{1}(f_{\alpha})}{h'}\xrightarrow{\zip{h'}{h}}
  \zipped{\pi_{1}(f_{\alpha})}{h }
  $$
  have all the same range provided $h'<h-\kappa$.
\end{lemma}

\begin{remark}
  The Mittag-Leffler condition is a condition on projective systems, and
  does not make sense for a single group. However, as observed in remark
  \ref{rk:Invariance}, the Novikov fundamental group comes with a
  preferred collection of essentially equivalent projective systems, and
  the statement claims they do all satisfy the Mittag-Leffler condition.
\end{remark}

\begin{proof}
  The invariance argument shows that this statement does not depend on
  the choice of the one form $\alpha$ nor on the primitive $f_{\alpha}$.

  More explicitly, let $f_{\alpha}$ and $f_{\beta}$ be two primitives of
  two $1$-forms $\alpha$ and $\beta$ in the cohomology class $u$, and
  suppose $\pi_{1}(f_{\alpha})$ satisfies the Mittag-Leffler condition
  for a constant $\kappa$. Recall that $K=\Vert
  f_{\alpha}-f_{\beta}\Vert_{\infty}<+\infty$. Then the following
  commutative diagram
  $$
  \xymatrix{
    \zipped{\pi_{1}(f_{\alpha})}{h}\ar[rrr]\ar[dr]&&&
    \zipped{\pi_{1}(f_{\alpha})}{h+2K+\kappa}\\
    &\zipped{\pi_{1}(f_{\beta})}{h+K}\ar[r]&
    \zipped{\pi_{1}(f_{\beta})}{h+K+\kappa}\ar[ur]\\
  }
  $$
  shows that $\pi_{1}(f_{\beta})$ satisfies the Mittag-Leffler condition for
  the constant $2K+\kappa$.

  So we can suppose without loss of generality that $\alpha$ is Morse,
  and consider the constant $\kappa$ already defined in
  \eqref{eq:gammaxHeightBound}~: for each index $0$ critical point $x$
  of $f_{\alpha}$, there is a path $\gamma_{x}$ that starts at $x$ and
  goes down to $-\infty$ (i.e.
  $\lim_{t\to+\infty}f_{\alpha}(\gamma_{x}(t))=-\infty$), and does not go
  higher than $f_{\alpha}(x)+\kappa$.

  In particular, if $\gamma:[0,1]\to\Cov{M}$ is a path that starts and
  ends below a level $h-\kappa$ (i.e. $f_{\alpha}(\gamma(0))\leq h$ and
  $f_{\alpha}(\gamma(1))\leq h$) with generic ends, we can flow its ends
  down to either $-\infty$ or an index $0$ critical point $x$, in which
  case we extend the resulting path by $\gamma_{x}$. Through this
  process, $\gamma$ is extended as a path $\bar{\gamma}:\R\to\Cov{M}$
  such that~:
  \begin{enumerate}
  \item
    $\lim_{t\to\pm\infty}f_{\alpha}(\bar\gamma(t))=-\infty$
  \item
    $\bar{\gamma}$ and $\gamma$ are the same above level $h$, more
    precisely $\bar\gamma:\R\to\zipped{\Cov{M}}{h}$ and
    $\gamma:\R\to\zipped{\Cov{M}}{h}$ (where $\gamma$ is extended away
    from $[0,1]$ as the constant map to the base point) are the same maps.
  \end{enumerate}

  In particular, this proves that if an element
  $[\gamma]\in\zipped{\pi_{1}(f_{\alpha})}{h}$ is in the image of
  $$
  \zipped{\pi_{1}(f_{\alpha})}{h-\kappa}
  \xrightarrow{\zip{h-\kappa}{h}},
  \zipped{\pi_{1}(f_{\alpha})}{h}
  $$
  it is also in the image of
  $$
  \zipped{\pi_{1}(f_{\alpha})}{h'}
  \xrightarrow{\zip{h'}{h}},
  \zipped{\pi_{1}(f_{\alpha})}{h}
  $$
  for all $h'\leq h-\kappa$.
\end{proof}

Moreover, the invariance argument also shows the pro-commutator
subgroup and the pro-abelianization group do not depend on the
choice of $\alpha$ or $f_{\alpha}$, so that we end up with a well
defined group
$$
\proabelian{\pi_{1}(\Cov{M},u)} = \proabelian{\pi_{1}(f_{\alpha})} =
\varprojlim\ \abelian{(\zipped{ \pi_{1}(f_{\alpha})}{h})}.
$$

\subsection{Hurewicz morphism}
The object of this section is to establish the following theorems.

In the case where $\Cov{M}$ is the minimal integration cover, we
have the following result~:
\begin{theorem}\label{thm:HurewiczMini}
  Let $u\in H^{1}(M,\R)$ be a cohomology class, and
  $\Cov{M}$ be its minimal integration cover.

  Then there is a surjective homomorphism
  $$
  \h: \pi_1(\Cov{M}, u) \to HN_{1}(M,u;\Z).
  $$
  The kernel of $\h$ consists of the pro-commutator subgroup of
  $\pi_{1}(\Cov{M},u)$. In particular, $\h$ induces an isomorphism
  $$
  \h: \proabelian{\pi_1(\Cov{M}, u)} \simeq HN_{1}(M,u;\Z).
  $$
\end{theorem}

For more general integration covers, we have the slightly weaker
result~:
\begin{theorem}\label{thm:HurewiczGeneral}
  Let $u\in H^{1}(M,\R)$ be a cohomology class, and $\Cov{M}$ be one of
  its integration cover.

  Then there is a surjective homomorphism
  $$
  \h: \pi_1(\Cov{M}, u) \to
  \varprojlim_{h}H_{1}(\Cov{M},\sublevel{\Cov{M}}{h}).
  $$
  The kernel of $\h$ consists of the pro-commutator subgroup of
  $\pi_{1}(\Cov{M},u)$. In particular, $\h$ induces an isomorphism
  $$
  \h: \proabelian{\pi_1(\Cov{M}, u)} \xrightarrow{\simeq}
  \varprojlim_{h}H_{1}(\Cov{M},\sublevel{\Cov{M}}{h}).
  $$
\end{theorem}

Theorem \ref{thm:HurewiczMini} follows from theorem
\ref{thm:HurewiczGeneral} and an improved description of
$\varprojlim_{h}H_{1}(\Cov{M},\sublevel{\Cov{M}}{h})$ in terms of
Novikov homology in the particular cases when $\Cov{M}$ is the
minimal integration cover.

\begin{proof}[Proof of theorem \ref{thm:HurewiczMini}
  from theorem \ref{thm:HurewiczGeneral}]

  Recall that the Novikov complex can be interpreted as a projective
  limit of relative complexes \cite{Sikorav1987}~:
  $$
  CN_{*}(M,u) = \varprojlim_{h}C_{*}(\Cov{M},\sublevel{\Cov{M}}{h})
  $$
  In \cite{Sikorav1987}, J.-C. Sikorav shows that the comparison of the
  homology of the projective limit and the projective limit of the
  homologies gives rise to the following short exact sequence
  \begin{equation}
    \label{eq:lim1}%
    0\to\lim\nolimits^{1}
    H_{*+1}(\Cov{M},\sublevel{\Cov{M}}{h}) \to HN_{*}(\Cov{M},u)\to
    \varprojlim_{h}H_{*}(\Cov{M},\sublevel{\Cov{M}}{h})\to 0.
  \end{equation}

  In the particular case where $\Cov{M}$ is the minimal integration cover
  of $u$, the deck transformation group $\Deck$ is isomorphic to $\Z^{r}$
  for some $r>0$, and M.~Usher shows in \cite[Theorem~1.3,
  p.~1583]{Usher2008} that a Novikov chain $c$ that is a boundary of a
  Novikov chain $d$, is the boundary of a chain $d'$ that goes no higher
  than the maximal height of $c$ plus some constant $M$.

  This implies that $H_{*}(\Cov{M},\sublevel{\Cov{M}}{h})$ satisfies the
  Mittag-Leffler condition, and hence that the $\lim^{1}$ term in
  \eqref{eq:lim1} vanishes. This ends the proof of theorem
  \ref{thm:HurewiczMini} assuming theorem
  \ref{thm:HurewiczGeneral}.
\end{proof}

\begin{proof}[Proof of theorem \ref{thm:HurewiczGeneral}]
  Pick a $1$-form $\alpha$ in the class $u$ and a primitive $f_{\alpha}$
  of $\alpha$ on $\Cov{M}$. Observe that at each regular level $h\in\R$,
  we have a Hurewicz morphism
  $$
  \zipped{\h}{h}: \zipped{\pi_1(f_{\alpha})}{h}
  \to H_{1}(\zipped{\Cov{M}}{h})
  =   H_{1}(\Cov{M},\sublevel{\Cov{M}}{h}).
  $$
  ($H_1(X;A)\simeq H_1(X/A)$ holds in general when $A$ is a deformation
  retract of a neighborhood of $A$ in $X$~; when $h$ is a regular level,
  this is the case for $A=\sublevel{\Cov{M}}{h}$ in $X=\Cov{M}$, since every point has
  a neighborhood that is either contained in $A$ or pushed back in $A$ by
  the flow).

  Moreover, this sequence of morphisms $\zipped{\h}{h}$ are compatible
  with the morphisms associated to increasing levels $h<h'$. In other
  words, the following diagram
  $$
  \xymatrix{
    \dots\ar[r]&
    \zipped{\pi_1(f_{\alpha})}{h }\ar@{>>}^{\zipped{\h}{h }}[d]\ar^{\zip{h}{h'}}[r]&
    \zipped{\pi_1(f_{\alpha})}{h'}\ar@{>>}^{\zipped{\h}{h'}}[d]\ar[r]&\dots
    \\
    \dots\ar[r]&
    H_{1}(\Cov{M},\sublevel{\Cov{M}}{h })\ar^{p_{*}}[r]&
    H_{1}(\Cov{M},\sublevel{\Cov{M}}{h'})\ar[r]&
    \dots
    }
  $$
  (where $p_{*}$ is the map induced by the projection
  $\zipped{\Cov{M}}{h} \xrightarrow{p} \zipped{\Cov{M}}{h'}$) is
  commutative. In particular, the sequence of maps $(\zipped{\h}{h})_{h}$
  defines a morphism
  $$
  \pi_{1}(f_{\alpha}) \xrightarrow{\h} \varprojlim H_{1}(\Cov{M},\sublevel{\Cov{M}}{h}).
  $$

  Moreover, at each level, we have the exact sequence
  $$
  0\to\Big[\zipped{\pi_{1}(f_{\alpha})}{h},\zipped{\pi_{1}(f_{\alpha})}{h}\Big]\to
  \zipped{\pi_{1}(f_{\alpha})}{h}\xrightarrow{\zipped{\h}{h}}
  H_{1}(\Cov{M},\sublevel{\Cov{M}}{h})\to 0.
  $$
  This defines in fact an exact sequence of projective systems, and since
  the first term satisfies the Mittag-Leffler condition, we obtain an
  exact sequence for the projective limits (see \cite[Theorem A.14]{Massey1978}). In
  particular, we obtain that $\h$ induces an isomorphism
  $$
  \proabelian{\pi_{1}(f_{\alpha})} \xrightarrow[\h]{\ \sim\ }
  \varprojlim{H_{1}(\Cov{M},\sublevel{\Cov{M}}{h})}.
  $$
\end{proof}

\section{Examples}

In this section we focus on manifolds of the form
$$
M=\T^{n}\sharp X
$$
obtained as the connected sum of a torus and some closed manifold
$X$ of dimension $n\geq 4$. For convenience we let $G=\pi_{1}(X)$.

The main example we have in mind is the case where
$X=S\times\S^{n-3}$, where $S$ is the Poincaré sphere, in which
case we have~:
\begin{equation}
  \label{eq:PresentationOfpi1S}%
  G = \pi_{1}(S) =  <a,b\, |\, a^{5}=b^{3}, a^{5}=(ab)^{2}>,
\end{equation}
but other cases will also be considered, and the discussion below
holds for general $X$.

Let $\T^{n}\xrightarrow{\theta}\S^{1}$ be the first coordinate on
the torus, and consider the cohomology class
$$
u=\pi^{\star}d\theta,
$$
where $\pi:M\to\T^{n}$ is a projection mapping $X$ to a point.

The minimal integration cover of $u$ is then the manifold
$$
(\R\times \T^{n-1})\underset{\Z}{\sharp} X
$$
obtained by performing the connected sum of the cylinder with a
fresh copy of $X$ at each lift of the point where the connected sum
took place on the torus.

\subsection{Basic recollection of Novikov homology for $\T^{n}\sharp X$ }
We denote by $HN_{\star}(M,u)$ the Novikov homology associated to
this covering of $M$. Algebraically, this comes down to considering
the group $\Deck = \pi_{1}(M)/\ker u\simeq \Z$ as defining a local
coefficients system, and use coefficients in the Novikov completion
$\Lambda=\Z((t))$ of the group ring $\Z[\Deck]=\Z[t,t^{-1}]$.
Namely,
$$
\Lambda =\{\sum_{-N}^{+\infty}a_{n}t^{n}, a_{n}\in\Z\}
$$
is the set of Laurent power series over $\Z$, and we consider the
complex
$$
CN_{\ast}(M,u) = \Lambda\otimes_{\Z[t,t^{-1}]}
C_{\ast}(M,\Z[t,t^{-1}])
$$
where $C_{\ast}(M,\Z[t,t^{-1}])$ is (for instance) the simplicial
chain complex of $M$ with local coefficients in
$\Z[\Deck]=\Z[t,t^{-1}]$.

Finally, the Novikov homology $HN_{\ast}(M,u)$ is the homology of
$CN_{\ast}(M,u)$.

\begin{proposition}\label{prop:HNexample}
  With the above notations, we have $HN_{0}(M,u)=0$, and for
  $i=1,2$~:
  $$
  HN_{i}(M,u)=H_{i}(X)\underset{\Z}{\otimes}\Lambda.
  $$
\end{proposition}

\begin{proof}
  Vanishing of $HN_{0}(M,u)$ whenever $u\neq 0$ is a
  classical and general feature of Novikov homology.

  To compute  $HN_{i}(M,u)$ ($i=1,2$)  we use  the  universal coefficient
  formula and the flatness of the Novikov ring $\Z ((t))$  as a
  $\Z[t,t^{-1}]$-module \cite[Theorem 1.8, p. 339]{Pajitnov2006}, cf.
  \cite[Appendix C]{LO1995}

  \begin{equation*}
    HN_i (M, u)= H_i (\Cov{M})\otimes _{\Z[t,t^{-1}]}  \Z((t)).
  \end{equation*}

  To compute $H_{i}(\Cov{M})$, we split $\Cov{M}$ as the union of
  \begin{itemize}
  \item
    $A=(\R\times \T^{n-1})\setminus \{ p_{k},\ k\in\Z\}$   where the
    $p_{k}$ are the lifts of the point where the connected sum was
    performed,
  \item
    $B=\sqcup_{k\in\Z}\,(X\setminus \{ p\})_{k}$ the disjoint union of
    $\Z$ punctured copies of $X$.
  \end{itemize}

  The overlapping region $A\cap B$ is the disjoint union of $\Z$ copies
  of $\S^{n-1}\times[0,1]$, which has vanishing homology groups in degrees
  $1$ and $2$ provided $n\geq4$.

  By the Mayer-Vietoris sequence theorem we have for  $i =1, 2$
  $$
  H_i (\Cov{M}) =  H_i (X)^{\Z} \oplus
  H_i (\R \times  \T^{n-1}).
  $$

  Since $\Deck=\Z$ acts on $A = \R \times  \T^{n-1}\setminus \{ p_{k},
  k\in\Z\}$ by translation along the $\R$ factor, the induced action on
  the torus factor is trivial, and hence for all $x\in
  H_{i}(\R\times\T^{n-1})\otimes\Lambda$, we have
  $$
  (1-t)x=0.
  $$
  Since $(1-t)$ is invertible in $\Lambda$, we conclude that $x=0$, and
  finally, observing that $H_i (X)^{\Z} = H_i (X)\otimes{\Z[t,t^{-1}]}$,
  we have~:
  $$
  HN_i (\Cov{M})
  =  (H_i (X)\otimes{\Z[t,t^{-1}]})\underset{\Z[t,t^{-1}]}\otimes\Lambda
  =  H_i (X)\underset{\Z}{\otimes}\Lambda
  $$
\end{proof}

The universal cover $\UCov{M}$ of $M$ can be described as a tree,
with two kinds $a$ and $b$ of vertices~:
\begin{itemize}
\item
  $a$-vertices are associated to copies of $\R^{n}\setminus\Z^{n}$ where
  the lattice corresponds to the lifts of the attaching point on the
  torus, and have one edge for each such point to a vertex of type $b$,
\item
  $b$-vertices are associated to copies of $\UCov{X}\setminus\{g\cdot
  p,g\in\pi_{1}(X)\}$, where $\UCov{X}$ is the universal cover of $X$,
  and the points removed are the lifts of the attaching point on $X$, and
  have one edge for each such lift to a vertex of type $a$.
\end{itemize}
The edges then correspond to the lifts of the annulus introduced in
the connected sum.

The  Novikov homology $HN_{*}(M,u;\Z[\pi_1(M)])$ on the universal
cover is less pleasant to describe than on the minimal integration
cover, but we will only use the following basic facts~:
\begin{proposition}\label{prop:HNtildeExample}
  The minimal number of generators of the first Novikov homology group
  $HN_{1}(M,u;\Z[\pi_{1}(M)])$ associated to the universal cover is at most
  $\mu(\pi_{1}(X))$, where $\mu(\pi_{1}(X))$ is the minimal number of
  generators of~$\pi_1(X)$.

  If $\pi_{2}(X)=0$, then the second Novikov homology group associated to
  the universal cover vanishes~:
  $$
  HN_{2}(M,u;\Z[\pi_{1}(M)])=0
  $$
\end{proposition}
\begin{remark}
  Notice that $\pi_{2}(X)\simeq H_{2}(\UCov{X};\Z)$ where $\UCov{X}$ is
  the universal cover of $X$, and is naturally a $\Z[\pi_{1}(X)]$-module.
\end{remark}

\begin{proof}
  Given a family $(g_{1},\dots,g_{\mu})$ that generates $\pi_{1}(X)$, for
  each $i$, consider a preferred lift $\gamma_{g_{i}}$ to $\UCov{X}$ of a
  loop in $X$ based at the attaching point $p$ representing $g_{i}$.
  Similarly, in $\R^{n}$ consider for $1\leq j\leq n$, the path
  $\gamma_{x_{j}}$, that describes the $[0,1]$-segment along the $j^{th}$
  coordinate axis.

  Finally, choose a preferred path $\delta:[0,+\infty)\to\R^{n}$ starting
  at the origin, such that $\lim_{t\to+\infty} x_{1}(\delta(t))=-\infty$,
  and that avoids $\Z^{n}$ for $t>0$, and use the suitable shifts of
  $\delta^{-1}$ and $\delta$ to extend each $\gamma_{g_{i}}$ and
  $\gamma_{x_{j}}$ into a path $\R\to\UCov{M}$ along which the coordinate
  $x_{1}$ runs from and to $-\infty$~: each such path then defines a
  Novikov $1$-cycle denoted by $\bar\gamma_{g_{i}}$ or
  $\bar\gamma_{x_{j}}$.
    Notice that $\bar\gamma_{x_{j}}$ and $t\cdot\bar\gamma_{x_{j}}$ are
    cobordant so that the Novikov homology class $[\bar\gamma_{x_{j}}]$
    vanishes in $HN_{1}(M,u~;\Z[\pi_{1}(M)])$.

  It is not hard to see now that every Novikov $1$-cycle $\sigma$ in
  $HN_{1}(M,u~;\Z[\pi_{1}(M)])$, is cobordant to a cycle $\sigma'$ of the
  form
  $$
  \sigma' = \sum_{\substack{g\in\pi_{1}(M)\\
      \gamma\in\{\gamma_{g_{i}}\}\cup\{\gamma_{x_{j}}\}}} n_{g}\ g\cdot\gamma.
  $$
  Moreover, $\sigma'$ being a cycle, the boundary extensions when
  replacing $\gamma$ by $\bar\gamma$ in the previous sum do not
  contribute, and we have
  $$
  \sigma' =
  \sum_{\substack{g\in\pi_{1}(M)\\
      \gamma\in\{\gamma_{g_{i}}\}\cup\{\gamma_{x_{j}}\}}} n_{g}\ g\cdot\bar\gamma
  = \sum_{_{\gamma\in\{\gamma_{g_{i}}\}\cup\{\gamma_{x_{j}}\}}}\lambda_{\gamma}\bar\gamma,
  $$
  where the $\lambda_{\gamma}$ are Novikov coefficients in
  $\Cov{\Lambda}$. Finally, since $[\bar\gamma_{x_{j}}]=0$ for each $j$, we obtain
  the desired result for the first Novikov homology group.

  \bigskip

  To describe the second Novikov homology group, it is convenient to use
  the cellular homology associated to a cell decomposition that is well
  suited to the connected sum. In  particular, we choose a cell
  decomposition of $M$ in which no $2$-cell meets the annulus
  $C=\S^{n-1}\times[0,1]$ introduced when performing the connected sum.

  As a consequence, the $2$-cells
  split in two disjoint families, according to the
  component $A=\T^{n}\setminus C$ or $B=X\setminus C$ of $M\setminus C$
  they are contained in.

  The collection of all the lifts of the cells of $M$  then defines a
  cell decomposition of the universal cover $\UCov{M}$. Since the
  complement in $\UCov{M}$ of the lifts of the annulus $C$ is the disjoint
  union $\sqcup_{\lambda}A_{\lambda}\sqcup_{\mu}B_{\mu}$ of copies of
  open sets $A$ and $B$ where
  \begin{itemize}
  \item
    $A=\R^{n}\setminus\Z^{n}$ is the universal cover $\R^{n}$ of $\T^{n}$
    with the lifts of the attaching point removed.
  \item
    $B=\UCov{X}\setminus\{\pi_{1}(X)\cdot p\}$ is the universal cover
    $\UCov{X}$ of $X$ with the lifts of the attaching point removed,
  \end{itemize}
  we can again split the collection of $2$-cells according to this
  collection of components.

  Consider now a Novikov $2$-cycle $s$ (using the cellular homology). It
  can be split according to the components its cells belong to~:
  $$
  s = \sum_{\lambda}s_{\lambda}+\sum_{\mu}s_{\mu}
  $$
  where $s_{\lambda}$ is supported in $A_{\lambda}$ and $s_{\mu}$ in
  $B_{\mu}$. Then because the $A_{\lambda}$'s and $B_{\mu}$'s are
  disjoint we have~:
  \begin{gather*}
     \forall\lambda, \partial s_{\lambda}=0
    \quad\text{ and }\quad
     \forall\mu, \partial s_{\mu}  =0.
  \end{gather*}
  Moreover, a primitive $f_{\alpha}$ of a $1$-form $\alpha$ in the class
  $u$ is bounded on $B_{\mu}$. This implies that $s_{\mu}$ is finite, and
  hence defines a $2$-cycle in $\UCov{X}$.

  Since $H_{2}(B_{\mu})=H_{2}(\UCov{X})=\pi_{2}(X)=0$, this implies that
  $s_{\mu}$ is a boundary~: there is a chain $\omega_{\mu}$ in $B_{\mu}$
  such that $s_{\mu}=\partial \omega_{\mu}$.

  The chain $\omega = \sum_{\mu}\omega_{\mu}$ is then a Novikov
  chain (above any level, there are only finitely many $\mu$ such
  that $s_{\mu}\neq0$, so that $\omega$ is finite above any level),
  and
  $$
  \partial \omega = \sum_{\mu}s_{\mu}.
  $$

  In particular, the homology class $[s]$ can be written as
  $$
  [s] = \sum_{\lambda} [s_{\lambda}]
  $$

  On the other hand, denote by $t_{1}\in\pi_{1}(M)$ the loop associated
  to the first coordinate in the torus. It acts on each $A_{\lambda}$ as
  the translation along $(1,0,\dots,0)$, and any Novikov $2$-cycle $\tau$
  in $A_{\lambda}$ is homologous to  $t_{1}\cdot\tau$ (this is clear in
  $\R^{n}$, and the points that are removed from $A_{\lambda}$ are
  sufficiently high codimensional to be avoided).

  As a consequence, for all $\lambda$~:
  $$
  t_{1}\cdot[s_{\lambda}] = [s_{\lambda}],
  $$
  (where the brackets denote the homology class). We conclude that
  $$
  (1-t_{1})\cdot [s_{\lambda}]=0,
  $$
  and since $(1-t_{1})$ is invertible
  in the Novikov ring, $[s_{\lambda}]=0$, and finally $[s]=0$.
\end{proof}

\subsection{Novikov fundamental group for $\T^{n}\sharp X$}

In this section we compute and investigate the Novikov fundamental
group associated to the minimal integration cover in the particular
situation described above where $M=\T^{n}\sharp X$. For convenience,
the fundamental group of $X$ will be denoted by $G$.

\subsubsection{Computation of the Novikov fundamental group}
The Novikov fundamental group $\pi_{1}(\T^{n}\sharp X,u)$ naturally
appears as a completion of the free product of infinitely many
copies of $G$. More precisely, consider the infinite sequence of
groups $(G_{k})_{k\in\Z}$ where all the groups $G_{k}$ are copies of
$G$, and define for each $h\in\Z$ the group
$$
\Pi_{h}=\freeprod_{k\geq h}G_{k}.
$$
For two integers $h,h'$ with $h<h'$, we have
$$
\Pi_{h'} = \Pi_{h}/(G_{k}=1, h\leq k <h').
$$
Hence there are projections
\begin{equation}
  \label{eq:formalzip}
\zip{h}{h'}:\Pi_{h}\to\Pi_{h'}
\end{equation}
such that for $h<h'<h''$, we have~:
$$
\zip{h}{h''}=\zip{h'}{h''}\circ\zip{h}{h'}.
$$

Define then
$$
\Pi = \varprojlim_{h}\Pi_{h}.
$$
\begin{proposition}
  With the above notations, $\pi_{1}(\Cov{M},u)=\Pi$.
\end{proposition}

\begin{proof}
  Pick a $1$-form $\alpha$ that ``separates'' the different copies of $X$,
  i.e. such that a primitive $f_{\alpha}$ has a level $h_{0}$, that can
  be supposed to be $0$, that does not touch any lift of $X$ nor of the
  annulus where the connected sum takes place.

  By rescaling the cohomology class $u$, we can suppose it has period
  $1$, i.e. that $\alpha(\pi_{1}(M))=\Z$, so that the integral levels of
  $f_{\alpha}$ separate the different lifts of $X$.

  Then for any $h\in\Z$, we have
  $$
  \zipped{\pi_{1}(f_{\alpha})}{h} =
  \pi_{1}(\zipped{\Cov{M}}{h}) =
  \freeprod_{k\geq h}\pi_{1}(X) = \Pi_{h}.
  $$

  Since  the projections $\zipped{\Cov{M}}{h}\to\zipped{\Cov{M}}{h'}$
  given by inclusion of sublevels induce the same maps $\zip{h}{h'}$ as
  in \eqref{eq:formalzip}, we derive $\pi_{1}(\Cov{M},u)=\Pi$.
\end{proof}

With the notations above we obviously have
\begin{equation*}
  \mu_{DTC}(\pi_{1}(\Cov{M},u))\leq \mu(\pi_{1}(X)),
\end{equation*}
where $\mu(\pi_{1}(X))$ is the minimal number of generators of
$\pi_{1}(X)$.

This estimate is straightforward, but the reverse inequality, which
we expect to hold, is surprisingly far from obvious. We will limit
ourselves to the following statement~:

\begin{proposition}\label{prop:noncyclic}
  Unless $\pi_{1}(X)$ is a cyclic group,
  $\mu_{DTC}(\pi_{1}(\Cov{M},u))\geq 2$. In particular, when
  $X=S\times\S^{n-3}$, we have $\mu_{DTC}(\pi_{1}(\Cov{M},u))=2$.
\end{proposition}
\begin{remark}
From the usual group theory point of view, deck transformations turn
a single generator ``up to DTC'' into an infinite collection of
generators. Therefore, telling if a group has more than one
generator up to DTC is far from obvious in general.

To emphasize this issue, consider the same question for the infinite
cartesian product $\Pi'=G^{\Z}$ instead of the free product~: when
$G$ is the group given in \eqref{eq:PresentationOfpi1S} for
instance, $\Pi'$ has a single generator up to shift while neither
$G$ nor $\Pi'$ are cyclic...
\end{remark}

The end of this section is dedicated to the proof of proposition
\ref{prop:noncyclic}.

\begin{definition}
  Given a collection of groups $G_1, \dots, G_N$, we say that a reduced
  sequence is a sequence $g_1, \dots, g_m$ such that for $1\leq i\leq m$
  we have $g_i\in G_k$ for some $k$, $g_i\neq 1$ and $g_i, g_{i+1}$ are
  not in the same factor $G_l$ for some $l$.

  We say that $g=g_1\dots g_m\in \freeprod_{1\leq k\leq N}G_{k}$ is in
  the normal form if $g_1,\dots, g_m$ is a reduced sequence. According to
  the Normal Form Theorem for free product of groups, every element of
  the free product $\freeprod_{1\leq k\leq N}G_{k}$ can be written in a
  unique way in a normal form.
\end{definition}

\begin{definition}
  An element $g$ in a free product of groups $\freeprod_{1\leq k\leq
    N}G_{k}$ is said to be a single letter (resp. non trivial single
  letter) if its normal form has length $\leq 1$ (resp. $1$), i.e. if $g\in\
  G_{k_{0}}\subset \freeprod_{1\leq k\leq N}G_{k}$ (resp. $g\in
  G_{k_{0}}\setminus\{1\}$) for some index $k_{0}$.

  Similarly, an element $g$ in $\Pi$ will be said to be a single letter
  if $g\in G_{k}\subset \Pi$ for some $k$.
\end{definition}

\begin{lemma}\label{lem:power_is_singleletter}
  Let $g\in\freeprod_{k=1}^{N}G_{k}$ be an element of a finite free
  product of groups. If $g^{\nu}$ is a single letter for some $\nu\neq0$, then $g$
  itself is a single letter.
\end{lemma}
\begin{proof}
  Let $g= g_{1}\dots g_{N}$ be in the normal form.

  Let $k$ be the maximal index such that $(g_{1}\dots g_{k}) =
  (g_{N-k+1}\dots g_{N})^{-1}$. Then letting $w=g_{1}\dots g_{k}$ and
  $g'=g_{k+1}\dots g_{N-k}$ (notice that $g'\neq 1$), we have $g = w g'
  w^{-1}$ and the normal form of $g^{\nu}$ is
  $$
  g^{\nu}=w\, {g'}^{\nu}\, w^{-1}.
  $$

  For $g^{\nu}$ to be a single letter hence requires that $w=1$ and $g'$
  to be a single letter, which means that $g$ itself is a single letter.
\end{proof}
\begin{remark}\label{rk:samefirstlastletters}
  Notice for future use that the proof also shows that if $g$ is not a
  single letter, the first and last letters of the normal form of any non
  trivial power of $g$ are the same as that of $g$ or $g^{-1}$.
\end{remark}

\begin{lemma}\label{lem:top_is_single}
  Suppose $g\in\Pi=\freeprod_{k\in\Z}G_{k}$ is such that the sub-group
  generated by $g$ up to DTC contains a non trivial single letter $a\in
  G_{k_{0}}\setminus\{1\}$.  Let $h_{0}$ be the first level above which
  $\zip{\infty}{h_{0}}(g)\neq 1$. Then the normal form of
  $\zip{\infty}{h_{0}}(g)$ is a single letter.
\end{lemma}

\begin{proof}
  By assumption, there are shifts of powers of $g$ such that above the
  level $k_{0}$ of this single letter $a$, we have~:
  $$
  a = \zip{\infty}{k_{0}}\left(\prod_{k=1}^{N} s^{d_{k}}(g^{\nu_{k}})\right),
  $$
  where $s$ denotes the upward shift. Since everything is defined up to
  shift, we can suppose that $d=\max\{d_{k},1\leq k\leq N\}=0$. Then
  $k_{0}\leq h_{0}$, so that
  \begin{equation}\label{eq:ziph0}%
    \zip{\infty}{h_{0}}\left(\prod_{k=1}^{N}
      s^{d_{k}}(g^{\nu_{k}})\right) = \zip{\infty}{h_{0}}(a)=
    \begin{cases}
      a &\text{ if } k_{0}=h_{0},\\
      1 &\text{ if } k_{0}<h_{0}.\\
    \end{cases}
  \end{equation}
  In both cases, $\zip{\infty}{h_{0}}(a)$ is a single,
  possibly trivial, letter $b$.

  But for all $d<0$, we have $\zip{\infty}{h_{0}}(s^{d}g)=1$, so that
  \eqref{eq:ziph0} reduces to
  \begin{equation*}
    \zip{\infty}{h_{0}} (g^{\nu}) = b = [\zip{\infty}{h_{0}} (g)]^{\nu},
  \end{equation*}
  where $\nu=\sum_{k/d_{k}=0}\nu_{k}$. From lemma
  \ref{lem:power_is_singleletter} it follows that
  $\zip{\infty}{h_{0}}(g)$ is a single letter, which is non trivial by the
  definition of $h_{0}$.
\end{proof}

\begin{proof}[Proof of proposition \ref{prop:noncyclic}]
  Suppose we can find a single generator $g$ for $\Pi$ up to DTC.

  Let $h_{0}$ be the first level above which
  $\zip{\infty}{h_{0}}(g)\neq 1$. From lemma \ref{lem:top_is_single} it follows that
  $$
  \zip{\infty}{h_{0}}(g)=s^{h_{0}}(a)
  $$
  is a single letter $a$ in level $h_{0}$. Here $G$ is identified with
  $G_{0}$.

  Let then $h_{1}$ be the first level above which the normal form of $g$
  contains a letter that is not a power of $a$~:
  \begin{gather*}
    \forall k\in \{h_{1}+1,\dots,h_{0}\},\
    \zip{\infty}{k}(g)\in \freeprod_{h_{1}<h} \big(s^{h}(a)\big)^{\Z},\\
      \zip{\infty}{h_{1}}(g)\notin \freeprod_{h_{1}\leq h} \big(s^{h}(a)\big)^{\Z}.
  \end{gather*}

  Let then
  \begin{equation*}
    \omega = \zip{\infty}{h_{1}}(g)\in \freeprod_{k\geq h_{1}} G_{k}.
  \end{equation*}
  Observe that $\omega$ is not a single letter, and hence none of its non
  trivial powers either. The normal form of $\omega$
  \begin{equation*}
    \omega = s^{k_{0}}(b_{0})s^{k_{1}}(b_{1})\dots s^{k_{N}}(b_{N})
  \end{equation*}
  is a sequence of letters $b_{i}\in G\setminus\{1\}$ in levels $k_{i}$
  such that $k_{i}\geq h_{1}$ and $|i-j|=1\Rightarrow k_{i}\neq k_{j}$.

  Let $A_{-}$ and $A_{+}$ be the longest possible words at the beginning
  and at the end of $\omega$ whose letters are all powers of $a$ in a
  level $<h_{0}$. More precisely, let~:
  $$
  A_{-}=\prod_{0\leq i< i_{-}}s^{k_{i}}(b_{i}),
  \quad
  \omega'=\prod_{i_{-}\leq i\leq i_{+}}s^{k_{i}}(b_{i}),
  \quad
  A_{+}=\prod_{i_{+}< i\leq N}s^{k_{i}}(b_{i}),
  $$
  where $i_{-}=\inf\{i, b_{i}\notin a^{\Z} \text{ or } k_{i}=h_{0}\}$ and
  $i_{+}=\sup\{i, b_{i}\notin a^{\Z} \text{ or } k_{i}=h_{0}\}$. Then
  \begin{equation}
    \label{eq:A-omegaA+} \omega = A_{-}\omega' A_{+}
  \end{equation}
  and the first and last letters of $\omega'$,
  $s^{k_{i_{-}}}(b_{i_{_{-}}})$ and $s^{k_{i_{+}}}(b_{i_{_{+}}})$ and are
  such that~:
  \begin{gather}
    \label{eq:clean_b0}%
    k_{i_{-}}\geq h_{0} \text{ or } \forall\alpha\in\Z, a^{\alpha}\neq
    1\Rightarrow a^{\alpha}b_{i_{-}}\neq 1,\\
    \label{eq:clean_bN}%
    k_{i_{+}}\geq h_{0} \text{ or } \forall\alpha\in\Z, a^{\alpha}\neq
    1\Rightarrow b_{i_{+}}a^{\alpha}\neq 1.
    \intertext{  Moreover, if $k_{i_{-}}= k_{i_{+}}$ and
      $\exists\alpha\in\Z, b_{i_{+}}a^{\alpha}b_{i_{-}}=1$, we replace
      $b_{i_{-}}$ by $a^{\alpha}b_{i_{-}}$ and $A_{-}$ by $A_{-}s^{k_{i_{-}}}(a^{-\alpha})$,
      so that  $b_{i_{+}}b_{i_{-}}=1$, and hence}%
    k_{i_{-}}\neq k_{i_{+}} \text{ or }
    \label{eq:clean_bNb0}%
    \forall\alpha\in\Z, a^{\alpha}\neq 1\Rightarrow
    b_{i_{+}}a^{\alpha}b_{i_{-}}\neq 1.
  \end{gather}
  Finally, recall for future use that from remark
  \ref{rk:samefirstlastletters}, the first and last letters $b_{i_{-}}$
  and $b_{i_{+}}$ of $\omega'$ are also the first and last letters of any non
  trivial power of $\omega'$.

  \medskip

  Since $G$ is not cyclic, there is an element
  \begin{equation}
    \label{eq:newelem}
    c\in G\setminus\{a^{\Z}\}
  \end{equation}
  which we regard as a single letter in $G_{0}\subset \Pi$. Since $g$
  generates $\Pi$ up to DTC, there is a word
  $x=\prod_{i}s^{d_{i}}(g^{\nu_{i}})$ such that
  $$
  \zip{\infty}{0}(x)=c\in G_{0}.
  $$
  Observe that $d = \max(d_{i})$ satisfies $d+h_{1}\geq 0$, since
  otherwise all the letters in $\zip{\infty}{0}(x)$ would belong to
  $a^{\Z}$, which contradicts \eqref{eq:newelem}. To reduce notations,
  and since everything is defined up to shift, we shift everything by
  $-d$, so that we can consider that $d=0$ and $c$ belongs to the level
  $-d$, which is not higher than $h_{1}$.

  Then the normal form of $\zip{\infty}{h_{1}}(x)$ is either $1$ (if
  $-d<h_{1}$) or the single letter $c$ (if $-d=h_{1}$)~:
  \begin{equation*}
    \zip{\infty}{h_{1}}(x)=1 \quad\text{ or }\quad
    \zip{\infty}{h_{1}}(x)=c\in G_{h_{1}}
  \end{equation*}

  On the other hand, $\zip{\infty}{h_{1}}(x)$ expands as
  $$
  \zip{\infty}{h_{1}}(x)=\prod_{i}\zip{\infty}{h_{1}}(s^{d_{i}}(g))^{\nu_{i}}
  $$
  and when $d_{i}<0$, the normal form of
  $\zip{\infty}{h_{1}}(s^{d_{i}}(g))$ is a word whose letters are all
  powers of $a$ in levels in $\{h_{1},\dots,h_{0}-1\}$, while for
  $d_{i}=0$, we have $\zip{\infty}{h_{1}}(s^{d_{i}}(g))=\omega$. As a
  consequence, we can write $\zip{\infty}{h_{1}}(x)$ as
  \begin{equation*}
    \zip{\infty}{h_{1}}(x) =
    A_{0}\,\omega^{\beta_{1}}A_{1}\dots
    \omega^{\beta_{N'}} A_{N'},
  \end{equation*}
  where $\omega^{\beta_{i}}\neq 1$ for $1\leq i \leq N'$ and $N'$ is
  indeed at least $1$ because of assumption \eqref{eq:newelem}.
  \begin{equation}
    \label{eq:A}%
    A_{i}\in\freeprod_{h_{1}\leq k<h_{0}}\Big(s^{k}(a)\Big)^{\Z}.
  \end{equation}
  Using \eqref{eq:A-omegaA+} and merging the $A_{i}$ and the possible
  $A_{-}$ or $A_{+}$, we can rewrite this as
  \begin{equation}
    \label{eq:nonreduced_xbis}%
    \zip{\infty}{h_{1}}(x) =
    A'_{0}\,\omega'^{\beta'_{1}}A'_{1}\dots
    \omega'^{\beta'_{N''}} A'_{N''},
  \end{equation}
  where the $A'_{i}$ also satisfy \eqref{eq:A} and
  $$
  A'_{i}\neq 1 \text{ for } 0< i<N''.
  $$

  Observe now that from \eqref{eq:clean_b0}, \eqref{eq:clean_bN}, and
  \eqref{eq:clean_bNb0}, the normal form of $\zip{\infty}{h_{1}}(x)$
  is obtained from the concatenation of the normal form of all the
  $A'_{i}$ and $\omega'^{\beta'_{j}}$ given by \eqref{eq:nonreduced_xbis} only
  by merging the first and/or last letter of $A'_{i}$ with the last and/or
  first letter of its neighbors if ever possible.

  In particular, the collection of levels supporting non trivial letters
  in the normal form of $\zip{\infty}{h_{1}}(x)$ is the union of the
  levels supporting non trivial letters in the normal form of all
  the $A'_{i}$ and $\omega'^{\beta'_{i}}$.

  Since none of the $\omega'^{\beta'_{i}}$ are single letters, this
  collection contains at least two levels and $\zip{\infty}{h_{1}}(x)$
  cannot be a single letter, which is a contradiction.
\end{proof}

\subsubsection{An upper bound for the deficiency}

Estimating $\mu_{DTC}$ and $\rho_{DTC}$ is hard in general, but the
object of this section is to give an upper bound for the deficiency
of $\Pi$ up to DTC, i.e. of the maximum of the difference $m-r$ of
the number of generators $m$ and relations $r$ in a finite
presentation up to
 DTC.

For the group $\pi_{1}(X)$ itself, the deficiency is bounded above
by the dimension of the real vector space
$$
L=\Hom(\pi_{1}(X),\R)
$$
of group morphisms from $\pi_{1}(X)$ to $(\R,+)$. The naive Novikov
counterpart of this space would be the space
$$
L_{\Lambda_{\R}} = \Hom^{\Deck}(\pi_{1}(\Cov{M},u),\Lambda_{\R})
$$
of shift equivariant morphisms from $\pi_{1}(\Cov{M},u)$ to the
field
$$
\Lambda_{\R}=\R((t))
$$
of Laurent formal series, seen as an additive group. However, we
need to keep track of the filtration and we restrict attention to
morphisms that are compatible with the inverse systems in the
following way.

First observe that $\Lambda_{\R}$ is in fact the inverse limit of
the inverse system
$$
\cdots \xrightarrow{[\cdot]_{d+1}}
\R_{d+1}((t))\xrightarrow{[\cdot]_{d}}
\R_{d}((t))\xrightarrow{[\cdot]_{d-1}}\cdots
$$
where $\R_{d}((t))$ denotes the Laurent polynomials of degree at
most $d$, and the transition maps are given by the truncation of
high degree terms. For convenience, we will write $\Lambda_{d} =
\R_{d}((t))$. Then~:
$$
\Lambda_{\R} = \varprojlim_{d\to +\infty}\Lambda_{d}.
$$

Let then
$$
L_{\Lambda_{\R}} = \ProHom^{\Deck}(\Pi, \Lambda)
$$
be the space of shift-equivariant morphisms from $\Pi$ to
$\Lambda_{\R}$ that are limits of morphisms from the inverse system
$(\Pi_{h})_{h\in\Z}$ to the inverse system $(\Lambda_{d})_{d\in\Z}$,
i.e. the space of morphisms $\phi:\Pi\to\Lambda_{\R}$ such that
there is a constant $d_{0}\in\Z$ and a collection
$(\phi_{h})_{h\in\Z}$ with
$\phi_{h}\in\Hom(\Pi_{h},\Lambda_{d_{0}-h})$, such that
\begin{enumerate}
\item
  the following diagram
  $$
  \xymatrix{
    \cdots\ar[r]&
    \Pi_{h-1}\ar^{\zip{h-1}{h}}[r]\ar^{\phi_{h-1}}[d]&
    \Pi_{h}\ar[r]\ar^{\phi_{h}}[d]&\cdots
    \\
    \cdots\ar[r]& \Lambda_{d_{0}-h+1}\ar^{[\cdot]_{d_{0}-h}}[r]& \Lambda_{d_{0}-h}\ar[r]&\cdots
  }
  $$
  is commutative,
\item
  each $\phi_{h}$ is shift equivariant~:
  $$
  \forall g\in\Pi_{h},\
  \phi_{h}(s(g)) = t^{-1}\phi_{h-1}(g),
  $$
\item
  and
$$
\phi = \varprojlim_{h}\phi_{h}.
$$
\end{enumerate}
Here, $\Pi_{h}=\freeprod_{k\geq h}G_{k}$ and $s$ denotes the map
$\Pi_{h-1}\to\Pi_{h}$ induced by the positive shift.

Given an element $\phi$ which is the limit of a morphism
$$(\phi_{h}:\Pi_{h}\to\Lambda_{d_{0}-h})_{h\in\Z}\in
L_{\Lambda_{\R}},$$ and $\lambda\in\Lambda_{\R}$, the morphisms
$$
\psi_{h}:
\begin{array}{ccc}
  \Pi_{h}&\to &\Lambda_{d_{0}-\nu-h}\\
        g &\mapsto& [\lambda\phi_{h}(g)]_{d_{0}-\nu-h}
\end{array}
$$
where $\nu$ is the $t$ valuation of $\lambda$, define a new element
$\psi = \lambda\phi$ in $L_{\Lambda_{\R}}$ that only depends on
$\phi$ and not on the choice of $d_{0}$ and $(\phi_{h})_{h}$.
Endowed with this operation, it is not hard to check that~:
\begin{proposition}
  The space $L_{\Lambda_{\R}}$ is a $\Lambda_{\R}$-vector space.
\end{proposition}

\bigskip

Recall that a finite presentation
$<g_{1},\dots,g_{m}|w_{1},\dots,w_{r}>$ of $G=\pi_{1}(X)$, allows to
identify $L$ as the kernel of a linear map
\begin{equation}
  \label{eq:MorphismsAsLinearSystem}%
  0\to L\to \R^{m}\to \R^{r} \to 0
\end{equation}
defined by the abelian rewriting of the words $w_{1},\dots,w_{r}$.

To establish a similar statement for $L_{\Lambda_{\R}}$, some
restriction on the family of generators is required. Observe that
the kernels $K_{h}$ of the evaluation maps
$\F_{h}\xrightarrow{\eval_{h}}\Pi_{h}$ defined in \eqref{eq:eval_h}
form an inverse system.

\begin{definition}
  A presentation up to DTC of $\pi_{1}(\Cov{M},u)$ is said to have well
  behaved relations if the inverse system $(K_{h})$ satisfies the Mittag-Leffler condition.
\end{definition}

This condition is not restrictive for the Morse interpretation of
the Novikov fundamental group since it always gives well-behaved
relations~:
\begin{proposition}
  The presentation of $\pi_{1}(\Cov{M},u)$ up to DTC associated to the
  index $1$ and $2$ critical points of a Morse $1$-form $\alpha$ in the
  class $u$ always has well behaved relations.
\end{proposition}

\begin{proof}
  The proof of statement \ref{prop:Crit2SpansRelations} shows that the
  maps $K_{h}\to K_{h'}$ are in fact surjective.
\end{proof}

We can now generalize $\eqref{eq:MorphismsAsLinearSystem}$ to the
Novikov
 situation.
\begin{lemma}
  Let $(g_{1},\dots,g_{m})$ and $(w_{1},\dots,w_{r})$ be a presentation
  up to DTC of $\Pi = \pi_{1}(\Cov{M},u)$ with well-behaved relations.
  Then there is a linear map $\Lambda_{\R}^{m}\to\Lambda_{\R}^{r}$ whose
  kernel is $L_{\Lambda_{\R}}$, i.e. there is a short exact sequence
  \begin{align}
  \label{eq:exactsequenceOLLambda} 0\to
    L_{\Lambda_{\R}}\xrightarrow{\epsilon}
    \Lambda_{\R}^{m}\xrightarrow{\rho} \Lambda_{\R}^{r}.
  \end{align}
  In particular, $\dim_{\Lambda_{\R}}L_{\Lambda_{\R}}\geq m-r$.
\end{lemma}

\begin{proof}
  The map $\epsilon$ is defined by evaluating a morphism $\phi\in
  L_{\Lambda_{\R}}$ on $(g_{1},\dots,g_{m})$~:
  $$
  \epsilon(\phi) = (\phi(g_{1}),\dots,\phi(g_{m})).
  $$
  To check it is injective, consider some $\phi\in \ker\epsilon$, and a
  sequence $$(\phi_{h}:\Pi_{h}\to\Lambda_{d_{0}-h})_{h\in\Z}$$ defining it.
  Since $\phi$ is supposed to be shift equivariant, we have then
  $\phi_h(\zip{\infty}{h}(s^{k}(g_i)))=0$ for all $i\in\{1,\dots,m\}$ and
  all $k\in\Z$. Since $(g_{1},\dots,g_{m})$ generates $\Pi$ up to DTC,
  this means that $\phi_{h}$ vanishes on the image of
  $\Pi\xrightarrow{\zip{\infty}{h}}\Pi_{h}$ and hence that $\phi=0$.

  \bigskip

  The definition of the map $\rho$ requires some extra care because of
  the completion process involved in the definition of
  $\Free[\Z]{\{g_{1},\dots,g_{m}\}}$. Recall that the deck transformation
  group $\Deck$ is isomorphic to $\Z$ in our case. The completion process
  rests on a notion of height for the letters, and in order to keep the
  notations reasonably simple, following remark \ref{rk:ArbitraryHeight},
  we fix the height of all the letters $g_{1},\dots,g_{m}$ to be $0$ and
  assume that the downward shift decreases the height by $1$. The
  associated restriction maps, in the completion of the free group, will
  still be denoted by $\zip{\infty}{h}$~:
  $$
  \Free[\Z]{\{g_{1},\dots,g_{m}\}} \xrightarrow{\zip{\infty}{h}}
  \zipped{\PreFree[\Z]{\{g_{1},\dots,g_{m}\}}}{h}.
  $$

  Let $w$ be one of the relations $w_{i}$, $1\leq i\leq r$. It is an
  element of the projective limit $\varprojlim
  \zipped{\PreFree[\Z]{\{g_{1},\dots,g_{m}\}}}{h}$, and hence, above each
  level $h$, $\zip{\infty}{h}{w}$ is a finite word in shifts of the
  letters $g_{1},\dots,g_{m}$. Using the notations of section
  \ref{sec:Relations}, we can write
  $$
  \zip{\infty}{h}(w) = \prod_{j=1}^{N}(k_{j},g_{i_{j}}^{\alpha_{j}}),
  $$
  where the shift components $k_{j}$ are such that $k_{j}\geq h$.

  Given an element $\lambda =
  (\lambda_{1},\dots,\lambda_{m})\in\Lambda_{\R}^{m}$, we define an
  element $\rho_{w}(\lambda)\in\Lambda_{\R}$ in the following way.

  If $\lambda=0$, we simply let $\rho_{w}(\Lambda)=0$. If not, consider
  an integer $\nu$ such that
  \begin{equation}
    \label{eq:nu}%
    \nu\leq \min\{\nu(\lambda_{i}) : \lambda_{i}\neq0 \},
  \end{equation}
  where $\nu(\lambda_{i})$ denotes the $t$-valuation of $\lambda_{i}$.
  For $h\in\Z$, define
  $$
  \rho_{w,h}(\lambda) =
  \big[\sum_{j=1}^{N}\alpha_{j}t^{-k_{j}}\lambda_{i_{j}} \big]_{\nu-h}\in\Lambda_{\nu-h}.
  $$

  Recall that for $h<h'$, we have
  $$
  \zip{\infty}{h'}(w)=\zip{h}{h'}(\zip{\infty}{h}(w)),
  $$
  i.e. that $\zip{\infty}{h'}(w)$ is obtained from $\zip{\infty}{h}(w)$
  by removing the letters whose levels are in $\{h,\dots,h'-1\}$. In
  particular, we have~:
  \begin{align*}
    [\rho_{w,h}(\lambda)]_{\nu-h'}
    &=
    \big[\sum_{\substack{1\leq j\leq N\\k_{j}<h'}}\alpha_{j}t^{-k_{j}}\lambda_{i_{j}}
    \big]_{\nu-h'} +
    \big[\sum_{\substack{1\leq j\leq N\\k_{j}\geq h'}}\alpha_{j}t^{-k_{j}}\lambda_{i_{j}}
    \big]_{\nu-h'}\\
    &=\big[\sum_{\substack{1\leq j\leq N\\k_{j}\geq h'}}\alpha_{j}t^{-k_{j}}\lambda_{i_{j}}
    \big]_{\nu-h'}\\
    &=\rho_{w,h'}(\lambda),
  \end{align*}
  which means that the sequence $(\rho_{w,h}(\lambda))$ defines an element
  $\rho_{w}(\lambda)$ in the inverse limit $\Lambda_{\R}$.

  Observe moreover that this limit does not depend on the choice of $\nu$
  provided $\nu$ satisfies \eqref{eq:nu}. In particular, an integer $\nu$
  that suits two different $\lambda$ can always be found, i.e. for
  $\lambda$, $\lambda'$ there exists
    \begin{equation*}
    \nu\leq \min\{\nu(\lambda_{i}), \nu(\lambda_{i}'): \lambda_{i}, \lambda_{i}'\neq0 \},
  \end{equation*}
  and it is not hard to check that the map
  $$
  \begin{array}{ccl}
    \Lambda_{\R}^{m}&\xrightarrow{\rho_{w}}& \Lambda_{\R}\\
    (\lambda_{1},\dots,\lambda_{m})&\mapsto&\rho_{w}(\lambda).
  \end{array}
  $$
  is $\Lambda_{\R}$ linear.

  \medskip
  We still have to prove the exactness of
  \eqref{eq:exactsequenceOLLambda}. Since each $w_{i}$ is a relation, we
  have $\rho\circ\epsilon=0$. Conversely, consider some
  $(\lambda_{1},\dots,\lambda_{m})\in\ker\rho$. Fix some integer $\nu$
  satisfying \eqref{eq:nu}, and consider the morphisms
  $$
  \begin{array}{ccc}
    F_{h}&\xrightarrow{\tilde{\phi}_{h}}&\Lambda_{\nu-h}\\
    \prod_{j}s^{k_{j}}(g_{i_{j}}^{\alpha_{j}})&\mapsto&
    \sum_{j}\alpha_{j}t^{-k_{j}}\lambda_{i_{j}}
  \end{array}.
  $$
  Since $\lambda\in\ker\rho$, we have
  $$
  \Rel_{h}\subset\ker \tilde{\phi}_{h},
  $$
  where $\Rel_{h}=\zip{\infty}{h}\Rel(A)$.

  Recall that, since the maps $\Pi_{h}\to\Pi_{h+1}$ are surjective and
  $A$ generates $\Pi$ up to DTC, every element $g\in\Pi_{h}$ can be written
  in the form
  \begin{equation}
    \label{eq:gdecomp}%
    g = \prod_{j=1}^{N}\zip{\infty}{h}(
    s^{k_{j}}(g_{i_{j}}^{\alpha_{j}})).
  \end{equation}
  In particular, this means that the projection $F_{h}\to\Pi_{h}$ is
  surjective.

  Notice however that although $\Rel_{h} \subseteq
  \ker(F_{h}\xrightarrow{\eval_{h}} \Pi_{h})$, the inclusion  might be
  strict in general, so that $\tilde{\phi}$ may not induce a morphism on
  $\Pi_{h}$ in general.

  The assumption that the relations are well behaved is made precisely to
  fix this~: let $\kappa$ be a constant such that
  $$
  \forall h'\in\Z, \forall h\in\Z,
  h \leq h'-\kappa \Rightarrow
  \zip{h}{h'}(K_{h}) = \zip{h'-\kappa}{h'}(K_{h'-\kappa}),
  $$
  where $K_{h}=\ker(F_{h}\xrightarrow{\eval_{h}}\Pi_{h})$. Then, for
  $h\leq h'-\kappa$, we have
  $$
  \zip{h}{h'}(K_{h})=\zip{\infty}{h'}(\Rel(A))=\Rel_{h'}.
  $$

  Observe now that $\zip{h-\kappa}{h}:\Pi_{h-\kappa}\to\Pi_{h}$ has a natural section
  $$
  \xymatrix{
    \Pi_{h-\kappa}\ar_{\zip{h-\kappa}{h}}[r]& \Pi_{h}\ar_{\iota}@/_2ex/[l]
  }
  $$
  mapping $\Pi_{h}$ to
  $\{1\}\ast\dots\ast\{1\}\ast\Pi_{h}\subset\Pi_{h-\kappa}$. Let
  $F'_{h-\kappa} = \eval_{h-\kappa}^{-1}(\iota(\Pi_{h})$ and $F''_{h} =
  \zip{h-\kappa}{h}(F'_{h-\kappa})$~:
  $$
  \xymatrix{
    F_{h-\kappa}\ar@{}[r]|-*[@]{\supset}\ar_{\eval_{h-\kappa}}[d]&
    F'_{h-\kappa}\ar^{\zip{h-\kappa}{h}}[r]\ar_{\eval_{h-\kappa}}[d]&
    F''_{h}\ar@{}[r]|-*[@]{\subset}&
    F_{h}\ar^{\eval_{h}}[d]\\
    \Pi_{h-\kappa}\ar@{}[r]|-*[@]{\supset}&
    \iota(\Pi_{h})\ar_{\zip{h}{h+\kappa}}[rr]&
    &
    \Pi_{h+\kappa}\ar_{\iota}@/_2ex/[ll]
  },
  $$
  and consider the restriction $\eval'_{h}$ of $\eval_{h}$ to $F'_{h}$.
  By construction, $\eval'_{h}$ is surjective, and its kernel satisfies~:
  $$
  \ker(\eval'_{h}) = \zip{h-\kappa}{h}(F'_{h-\kappa}\cap K_{h-\kappa})
  \subset\Rel_{h}.
  $$

  As a consequence, the morphism $\tilde{\phi}_{h}:F''_{h}\to\Lambda_{\nu-h}$ satisfies
  $$
  \ker(\eval'_{h})\subset\Rel_{h}\subset\ker(\tilde{\phi})
  $$
  and hence induces a morphism $\phi_{h}:\Pi_{h} =
  F''_{h}/\ker\eval'_{h}\to\Lambda_{\nu-h}$. Moreover, by
  construction we have
  $$
  \phi_{h}(g_{i}) = \zip{\infty}{\nu-h}(\lambda_{i}).
  $$
  As a consequence, the collection of morphisms $(\phi_{h})_{h\in\Z}$
  defines an element $\phi\in L_{\Lambda_{\R}}$ such that
  $$
  \lambda = \epsilon(\phi).
  $$
\end{proof}

\begin{proposition}\label{prop:LLambda=LotimesLambda}
  We have $L_{\Lambda_{\R}}=L\otimes\Lambda_{\R}$. In particular,
  $$
  \dim_{\R}L = \dim_{\Lambda_{\R}}L_{\Lambda_{\R}}.
  $$
\end{proposition}

\begin{proof}
  Let $\phi$ be a morphism in $L_{\Lambda_{\R}}$. Observe that a morphism
  $\phi=(\phi_{h})_{h\in\Z}\in L_{\Lambda_{\R}}$ is fully determined by
  its behaviour on single letters in $G_{0}\subset\Pi$, i.e. by the maps
  $$
  \phi_{h,0}={\phi_{h}}_{|_{G_{0}}}: G_{0}\to \Lambda_{h}.
  $$

  Splitting $\phi_{h,0}$ according to the degree of $t$, we can write
  $$
  \phi_{h,0}=\sum_{d=-N}^{-h} t^{d}\alpha_{h,d},
  $$
  where the $\alpha_{h,d}$ are now morphisms from $G_{0}$ to $\R$, i.e.
  elements of $L$.

  Splitting each $\alpha_{h,d}$ according to a basis
  $(\alpha^{(1)},\dots, \alpha^{(k)})$ of $L$, we write $\phi_{h,0}$ as
  $$
  \phi_{h,0}=\sum_{i=1}^{k} \lambda_{i,h} \alpha^{(i)},
  $$
  where each $\lambda_{i,h}$ is in $\Lambda_{-h}$. Moreover, for $h<h'$,
  $\lambda_{h'}=[\lambda_{h}]_{-h'}$, so that the sequence
  $\lambda_{i,h}$ defines an element $\lambda_{i}\in\Lambda_{\R}$, and we
  have
  $$
  \phi=\sum_{i=1}^{k} \lambda_{i}\alpha^{(i)}.
  $$
  Here $\alpha^{(i)}$ is seen as the element of $L$
  naturally induced by $\alpha^{(i)}$ on the $0$-level. This ends the
  proof of \ref{prop:LLambda=LotimesLambda}.
\end{proof}

\begin{corollary}
  The real dimension $\dim_{\R} \Hom(\pi_{1}(X),\R)$ is an upper bound for
  the deficiency up to DTC of $\pi_{1}(\Cov{M},u)$. Here the deficiency
  up to DTC is the maximum of $m-r$, where $m$ is the cardinal of a
  generating family, and $r$ of the associated relations, both up to DTC.
\end{corollary}

\subsection{Comparing Novikov fundamental group and  homologies}

In this section, we use the results of the two previous sections to
provide examples where the numerical estimates for the number of
index $1$ and $2$ critical points obtained from the Novikov
fundamental group are essentially different from the estimates that
can be derived from the Novikov homologies.

\subsubsection{Comparison on the minimal integration cover}

Suppose now that $X$ satisfies the following conditions
\begin{gather*}
  \dim X=n\geq4,\label{eq:ExampleCond1}\\
  H_{1}(X,\Z)=H_{2}(X,\Z)=0\label{eq:ExampleCond2},\\
  \pi_{1}(X)\neq 1\label{eq:ExampleCond3}.
\end{gather*}

This is the case for instance if $X=S\times\S^{n-3}$, $n\geq 4$,
where $S$ is the Poincaré homology sphere.

Recall that the manifold under interest is then $M=\T^{n}\times X$,
endowed with the one form $u=[\pi^{*}(d\theta_{1})]$, and $\Cov{M}$
is the minimal integration cover of $u$.

Then from proposition \ref{prop:HNexample}, we get
$$
HN_{1}(\Cov{M},u) = HN_{2}(\Cov{M},u) = 0,
$$
while from proposition \ref{prop:noncyclic} we get
$$
\pi_{1}(\Cov{M},u)\neq 1.
$$

In particular, the Novikov homology on the minimal cover does not
detect any index $1$ or $2$ critical points, while the Novikov
fundamental group does since we necessarily have
$$
\mingen(\Pi)>0.
$$
Similarly, if $G$ has non trivial relations, which is the case for
the Poincaré sphere for instance, then so does $\pi_{1}(\Cov{M},u)$
and hence
$$
\minrel(\Pi)>0.
$$

In the case where $X=S\times\S^{n-3}$, we can be even more precise,
since $\pi_{1}(X)$ is not cyclic and has at least $2$ generators, we
have
$$
\mu_{DTC}(\pi_{1}(\Cov{M},u)) = 2.
$$
Similarly \eqref{eq:PresentationOfpi1S} is a presentation of
$\pi_{1}(X)$ with $2$ generators and $2$ relations, so
$\rho_{DTC}(\pi_{1}(\Cov{M},u))\leq 2$, while it is not hard to see
from these relations that $\dim_{\R}\Hom(\pi_{1}(X),\R)=0$, so that
$\rho_{DTC}(\pi_{1}(\Cov{M},u))\geq 2$ and finally
$$
\rho_{DTC}(\pi_{1}(\Cov{M},u)) = 2.
$$

This proves Theorem \ref{thm:exple} which we recall below
\begin{corollary}\label{cor:NonTrivialPi1}
  When $X=S\times\S^{n-3}$, the Novikov fundamental  group associated to
  the class $u=\pi^{*}[d\theta]$ on $M=\T^{n}\sharp X$ is non trivial,
  and  any Morse $1$-form in this cohomology class  has to have at least
  $2$ index $1$ and $2$ index $2$ critical points.
\end{corollary}

This example shows that the lower bounds \eqref{eq:Crit1LowerBound}
and \eqref{eq:Crit2LowerBound} derived from Novikov fundamental
group are essentially different from the Novikov inequalities
derived from the Novikov homology associated to the minimal
integration cover.

\begin{remark}
  When $X=S\times\S^{n-3}$, we also have $\pi_{2}(X)=0$, and hence
  $HN_{2}(M,u;\Z[\pi_{1}(M)])=0$, so that the Morse-Novikov inequalities
  for the universal cover do not detect any index $2$ critical points
  either.

  However, the comparison of the Novikov homologies on the universal and
  minimal integration covers still allow to detect index $2$ critical
  points as detailed below.
\end{remark}

\subsubsection{Comparison on the universal cover}\label{sec:ComparisonOnUniversalCover}

A nice observation pointed to us by A. Pajitnov allows to detect
index $2$ critical points from the comparison of the Novikov
homology on the minimal and universal covers, even though the second
homology group vanish~: starting with a $1$-form $\alpha$ in the
class $u$, it is possible to modify to remove all the index $0$
critical points, without creating new index $2$ critical points.
Then, since $HN_{1}(M,u;\Z[\pi_{1}])\neq 0$, $\alpha$ has to have
index $1$ critical points. Moreover, since $HN_{1}(M,u)=0$, all
these points have to be killed by index $2$ critical points.

The number of index $2$ critical points detected by this observation
is at most that of index $1$ critical points. More precisely, it is
at most the difference $|\tilde{\beta}_{1}-\beta_{1}|$ where
$\beta_{1}$ and $\tilde{\beta}_{1}$ are the minimal number of
generators for $HN_{1}(M,u)$ and $HN_{1}(M,u;\Z[\pi_{1}(M)])$.

\medskip

Consider the case where $M=\T^{n}\sharp X$ with $X=\R\P^{n}$ ($n\geq
4$).

From \eqref{prop:HNexample}, we obtain
$$
HN_1 (M,u) =  \Z/2\Z((t)) \quad\text{ and }\quad HN_2 (M,u) =  0.
$$
On the other hand, $HN_{1}(M,u;\Z[\pi_{1}(M)])\neq 0$ so that from
\ref{prop:HNtildeExample}, we get that the minimal number of
generators of $HN_{1}(M,u;\Z[\pi_{1}(M)])$ is $1$, and
$$
HN_2 (M,u;\Z[\pi_{1}(M)]) =  0.
$$

In particular, the comparison of the Novikov homologies associated
to the minimal and universal covers does not provide any index $2$
critical points, while the Novikov fundamental group does, since we
obviously have
$$
\mu_{DTC}(\pi_{1}(M,u))=1 \text{ and } \rho_{DTC}(\pi_{1}(M,u))=1.
$$
In particular, every Morse $1$-form on $\T^{n}\sharp\R\P^{n}$ has to
have at least one index $2$ critical point.

{\bf Acknowledgement}.  HVL  thanks   JFB  and  Nguy\^en Ti\^en
Zung, L\^e Ngoc Mai for their  invitation to  Toulouse  and their
hospitality during her visit    in May 2016,    when  she and JFB
started this project aiming to  develop    ideas in
\cite{Barraud2014}   further for the Novikov  case.

AG thanks JFB and CIMI for the invitation to Toulouse in June 2017,
when she joined the project, the Hausdorff Research Institute
for Mathematics (HIM), University of Bonn for its support and
hospitality in the fall 2017, and the Mathematics Institute of 
Uppsala University for its nice working atmosphere.

We thank  Andrei Pajitnov for his  stimulating comments to the first
e-print of our paper.

\end{document}